\documentclass[reqno,11pt]{amsart}

\usepackage[latin1]{inputenc} % ±ØÐë¼ÓÕâ¸ö£¬·ñÔòÍ¶¸å±àÒë²»³É¹¦
\usepackage{listings, tcolorbox}% ±ØÐë¼ÓÕâ¸ö£¬·ñÔòÍ¶¸å±àÒë²»³É¹¦
\usepackage{amssymb, amsmath, amsthm}
\usepackage[colorlinks=true,linkcolor=blue,citecolor=red]{hyperref}
\usepackage{color}
\usepackage{epsfig}
\usepackage{amsmath}
\usepackage{amssymb}
\usepackage{amscd}
\usepackage{graphicx}
\usepackage{mathrsfs}
\usepackage{enumerate}
\usepackage{tikz}
\usepackage{epstopdf}
\usepackage{subfigure}
\usepackage[T1]{fontenc}
\numberwithin{equation}{section}

   \newtheorem{thm}{Theorem}[section]
   
   \newtheorem{prop}{Proposition}[section]
\newtheorem{rem}[thm]{Remark}

\newtheorem{RHP}[thm]{Riemann-Hilbert Problem}
\newtheorem{lem}[thm]{Lemma}
\newtheorem{cor}[thm]{Corollary}

\usepackage{epsfig}

\numberwithin{equation}{section}
\numberwithin{prop}{section}
\numberwithin{lemma}{section}
\numberwithin{re}{section}
\numberwithin{coro}{section}



\thanks{ Email: sftian@cumt.edu.cn, shoufu2006@126.com (S. F. Tian). }

\begin{document}

\title[On the long-time asymptotic of the CH equation]{On the long-time asymptotic behavior  of the Camassa-Holm equation in space-time solitonic regions}

%\today

\author[Li]{Zhi-Qiang Li}

\author[Tian]{Shou-Fu Tian$^{*}$}
\address{Zhi-Qiang Li, Shou-Fu Tian (Corresponding author) and Jin-Jie Yang \newline
School of Mathematics, China University of Mining and Technology, Xuzhou 221116, People's Republic of China}
\thanks{$^{*}$Corresponding author(sftian@cumt.edu.cn, shoufu2006@126.com).
}%\email{zqli@cumt.edu.cn (Z.Q. Li)}
\email{sftian@cumt.edu.cn, shoufu2006@126.com (S.F. Tian)}
%\email{jinjieyang@cumt.edu.cn (J.J. Yang)}

\author[Yang]{Jin-Jie Yang}

\begin{abstract}
%\textcolor{red}
{In this work, we are devoted to study
the Cauchy problem of the Camassa-Holm (CH) equation with  weighted Sobolev initial data in space-time solitonic regions
\begin{align*}
m_t+2\kappa q_x+3qq_x=2q_xq_{xx}+qq_{xx},~~m=q-q_{xx}+\kappa,\\
q(x,0)=q_0(x)\in H^{4,2}(\mathbb R),~~x\in\mathbb R, ~~t>0,
\end{align*}
where $\kappa$ is a positive constant.
Based on the Lax spectrum problem, a Riemann-Hilbert problem corresponding to the original problem is constructed to give  the solution of the CH equation with the initial boundary value condition.
Furthermore, by developing  the $\bar{\partial}$-generalization of Deift-Zhou nonlinear steepest descent method, different long-time asymptotic expansions of the solution $q(x,t)$ are derived. Four asymptotic regions are divided in this work: For $\xi\in\left(-\infty,-\frac{1}{4}\right)\cup(2,\infty)$, the phase function $\theta(z)$ has no stationary point on the jump contour, and the asymptotic approximations can be characterized with the soliton term confirmed by $N(j_0)$-soliton on discrete spectrum  with residual error up to $O(t^{-1+2\tau})$; For $\xi\in\left(-\frac{1}{4},0\right)$ and $\xi\in\left(0,2\right)$,  the phase function $\theta(z)$ has four and two stationary points on the jump contour, and the asymptotic approximations can be characterized with the soliton term confirmed by $N(j_0)$-soliton on discrete spectrum and the $t^{-\frac{1}{2}}$ order term on continuous spectrum  with residual error up to $O(t^{-1})$. Our results also confirm the soliton resolution  conjecture for the CH equation with weighted Sobolev initial data in space-time solitonic regions.
}
\end{abstract}

\maketitle

\tableofcontents

\section{Introduction}

In this work, we study the long-time asymptotic behavior for the initial value problem for the Camassa-Holm (CH) equation:
\begin{align}\label{CH-equation}
m_t+2\kappa q_x+3qq_x=2q_xq_{xx}+qq_{xx},~~m=q-q_{xx}+\kappa,\\
q(x,0)=q_0(x)\in H^{4,2}(\mathbb R),~~x\in\mathbb R, ~~t>0,
\end{align}
where $\kappa$ is a positive constant related to the critical shallow water wave speed, and   $q=q(x,t)$  is the fluid velocity in the $x$ direction.
%a real-valued function of   $x$ and $t$.
The CH equation \eqref{CH-equation} is different from the nonlinear Schr\"{o}dinger type equation which have   been extensively researched and has a range of excellent work \cite{Miller-1,Fokas-1,Fokas-2,Tian-PAMS,Tian-JDE,Tian-PA,Wangds-2019-JDE,Yang-Hirota}.
At first, the Camassa and Holm in \cite{CH-Lax} and Camassa, Holm and Hyman in \cite{CH94-AAM} proposed that the equation \eqref{CH-equation} can be applied to describe shallow water waves. It should be pointed out that the CH equation \eqref{CH-equation} has interesting physical meaning. Johson in \cite{Johnson-Proc-London} and Constantin and Lannes in \cite{CH-Constantin-ARMA} studied the Hydrodynamical relevance of the equation \eqref{CH-equation}.
Meanwhile, it has  rich mathematical structure, for example, it is bi-Hamiltonian and completely integrable \cite{CH-Lax,CH-Qiao-CMP}. Therefore, various researches on CH equation \eqref{CH-equation} are reported, including local and global well-posedness of its Cauchy problem, blow-up solutions, soliton-type solutions and other works \cite{CH-Constantin-ARMA-2,CH-Constantin-CPAM,CH-Constantin-AIF}. Moreover, Boutet de Monvel et al. have done many meaningful work on the the equation \eqref{CH-equation}, including the Riemann-Hilbert (RH) problem, long-time asymptotic behavior and Painlev\'{e}-type asymptotic  for the CH equation \cite{Boutet-1,Boutet-2,CH-Longtime,CH-Painleve}.

The aim of  this work is to employ $\bar{\partial}$-steepest descent method to investigate the soliton resolution conjecture of the CH equation \eqref{CH-equation} with the initial value condition
\begin{align*}
  q(x,0)=q_0(x)\in H^{4,2}(\mathbb R),~~x\in\mathbb R, ~~t>0,
\end{align*}
where $H^{k,2}(\mathbb{R})$ are defined in Proposition \ref{H-r-space}.

Inspired by earlier work of  Manakov, who have done significant contribution on the research of the long time asymptotic behavior of nonlinear evolution equations by using using the inverse scattering method \cite{Manakov-1974}, many researchers continuously follow his step  and have done a series of excellent work \cite{Grava-Siam,GM-2020,Kitaev-IP,K2,JDE-CH}. Then,  Zakharov and Shabat \cite{Zakharov-Shabat-1974} did a series work to develop the RH method by combining  inverse scattering method, and the RH method have been widely applied \cite{Boutet-y-1}.
%Since Manakov first paid attention to the long time asymptotic behavior of nonlinear evolution equations \cite{Manakov-1974}, the research of it has been widely concerned\cite{GM-2020}-\cite{K2}.
In 1976, Zakharov and Manakov derived the long time asymptotic solutions of NLS equation with decaying initial value \cite{Zakharov-1976}. In 1993, nonlinear steepest descent method  was proposed  by Defit and Zhou   %which can be used
to systematically study the long time asymptotic behavior of nonlinear evolution equations \cite{Deift-1993}. Over the years, the the nonlinear steepest descent method has been improved and developed . %After years of unremitting research by scholars, the nonlinear steepest descent method has been improved.
A notable  example is that as the initial value is smooth and decays fast enough, the error term is $O(\frac{\ln t}{t})$ shown in \cite{Deift-1994-1, Deift-1994-2}. Furthermore, in 2003, Deift and Zhou \cite{Deift-2003} have shown  that the error term is $O(t^{-(\frac{1}{2}+\iota)})$ for any $0<\iota<\frac{1}{4}$ when the initial value belongs to the weighted Sobolev space.

Recently, combining steepest descent with $\bar{\partial}$-problem, McLaughlin and Miller \cite{McLaughlin-1, McLaughlin-2} came up with a new idea, i.e., $\bar{\partial}$-steepest descent method, to investigate the asymptotic of orthogonal polynomials. Then, this method was successfully used to investigate defocusing NLS equation with finite mass initial data \cite{Dieng-2008} and with finite density initial data \cite{Cuccagna-2016}. %It should be pointed out that
Notably, compared with the nonlinear steepest descent method, there are some advantages, for example, the delicate estimates involving $L^{p}$ estimates of Cauchy projection operators can be avoided by using $\bar{\partial}$-steepest descent method. Also, the work in \cite{Dieng-2008} shows that the error term is $O(t^{-\frac{3}{4}})$ when the initial value belongs to the weighted Sobolev space. Therefore, a series of great work has been done by applying $\bar{\partial}$-steepest descent method \cite{AIHP,Faneg-2,Miller-2,Jenkins,Jenkins2,Li-CSP,Li-WKI,Fan-mKdV-Dbar-1,Fan-SP-Dbar}.

In this work, our main purpose is to study
the Cauchy problem of the CH equation \eqref{CH-equation} with  weighted Sobolev initial data $q_0(x)\in H^{4,2}(\mathbb R)$  in space-time solitonic regions.
Using a  simple  transformation
\begin{align}
		x=\tilde{x}, \ \
		t=\frac{1}{\kappa} \tilde{t},  \ \
		q(x, t)=\kappa \tilde{q}(\tilde{x}, \tilde{t}),
\end{align}
the CH equation \eqref{CH-equation} becomes
\begin{equation}
\tilde{m}_{\tilde{t}}+2 \tilde{q}_{\tilde{x}} +3\tilde{q}\tilde{q}_{\tilde{x}}= 2\tilde{q}_{\tilde{x}}\tilde{q}_{\tilde{x}\tilde{x}}+\tilde{q}\tilde{q}_{ \tilde{x}\tilde{x}\tilde{x}},~~ \tilde{m}=\tilde{q}-\tilde{q}_{\tilde{x}\tilde{x}}+1.
\end{equation}
So without loss of generally, we fix $\kappa=1$.

In our results, with the weighted Sobolev initial data $q_0(x)\in H^{4,2}(\mathbb R)$, we derive the leading order asymptotic approximation for the CH equation \eqref{CH-equation} (see Theorem \ref{Thm-1} in Section \ref{Asy-appro-CH}):
\begin{itemize}
  \item For $\xi\in\left(-\infty,-\frac{1}{4}\right)\cup(2,\infty)$,
  \begin{align*}%\label{result-1}
    q(x,t)=q(x,t|D_{j_0})+O(t^{-1+2\tau});
  \end{align*}
  \item For $\xi\in\left(-\frac{1}{4},2\right)$
  \begin{align*}%\label{result-1}
    q(x,t)=q(x,t|D_{j_0})+f_1t^{-1/2}+O(t^{-1}).
  \end{align*}
\end{itemize}
Compared with the results in \cite{Monvel-CH-Long-time}, our results are different.

%\noindent \textbf{Main results and remarks}
%
%Our main results and remarks are presented as follows.

%\noindent \textbf{Plan of the proofs}
\noindent \textbf{Organization of the rest of the work}

In Section 2,  based on the Lax pair, the RH problem for $M(z)$ is constructed for the CH equation \eqref{CH-equation}  with initial problem..

In Section 3, in order to obtain a a standard RH problem for $M^{(1)}(z)$ with its jump matrix able to be %can be
decomposed into two triangle matrices,  according to the oscillation term $\exp(2izy-\frac{t}{4z})$ in RH problem \ref{RH-2} and applying its sign distribution diagram shown in Fig. \ref{fig-1}, we introduce the matrix function $T(z)$.

In Section 4, by using the constructed matrix function $T(z)$, we define  a new   RH problem for $M^{(1)}(z)$  which admits a regular discrete spectrum and two triangular decompositions of the jump matrix.

In Section 5, we make the continuous extension of the jump matrix off the real axis by introducing a matrix function $R^{(2)}(z)$ and get a mixed $\bar{\partial}$-Riemann-Hilbert(RH) problem for $M^{(2)}(z)$.

In Section 6, we decompose the mixed $\bar{\partial}$-RH problem into two parts that are a model RH problem with $\bar{\partial}R^{(2)}=0$ and a pure $\bar{\partial}$-RH problem with $\bar{\partial}R^{(2)}\neq0$,  i.e., $M^{(2)}_{RHP}$ and $M^{(3)}$.

In Sections 7 and 8,  we solve the model RH problem $M^{(2)}_{RHP}$  via an outer model $M^{R}(z)$ for the soliton part and inner model $M^{(loc)}$ near the phase point $\xi_k$ which can be solved by matching  a model problem on a cross when $\xi\in\left(-\frac{1}{4},2\right)$. While, for $\xi\in\left(-\infty,-\frac{1}{4}\right)\cup(2,\infty)$, $M^{(2)}_{RHP}(z)=E(z)M^{R}(z)$.  Also, the error function $E(z)$ with a small-norm RH problem is computed.

In Section 9, the pure $\bar{\partial}$-RH problem for $M^{(3)}$ is studied.

Finally,  we obtain the soliton resolution conjecture and long time asymptotic behavior of the CH equation \eqref{CH-equation}.

%The finally appendix presents some detailed calculation proof.

\section{The Riemann-Hilbrt problem of CH equation}

This section will  give  the established Riemann-Hilbert problem corresponding the initial value problem for the  CH  equation \eqref{CH-equation} reported by Boutet
 de Monvel, Its and Shepelsky \cite{CH-Longtime,CH-Painleve}.

The Lax representation of the CH equation \eqref{CH-equation} has the following form in \cite{CH-Lax}
\begin{align}\label{Lax-O}
\begin{split}
\psi_{xx}=\frac{1}{4}\psi-\left(z^2+\frac{1}{4}\right)m\psi,\\
\psi_t=\left(-\frac{2}{1+4z^2}-q\right)\psi_x+\frac{1}{2}q_x\psi,
\end{split}
\end{align}
where $\psi:=\psi(x,t,z)$ and $z$ is the spectrum paremeter.

Then, the RH problem corresponding to the initial value problem for the CH equation \eqref{CH-equation} is given as follows.

\begin{prop}\label{Define-previous}
(see \cite{CH-Longtime,CH-Painleve})
Let $r(z)$ and $\mathcal{Z}=\{z_j, \gamma_j\}_{j=1}^{N}(Re z_j=0,~0<Im z_j<\frac{1}{2},~\gamma_j>0)$ be the reflection coefficient and the discrete spectrum with the norming constants $\gamma_j$, respectively, corresponding to the initial data $q_0(x)$ via the scattering problem for the spectrum problem \eqref{Lax-O}.\\
Let
\begin{align*}
  M(y,t;z)=( M_1(y,t;z) M_2(y,t;z))
\end{align*}
be the solution of the following RH problem in the complex $z$-plane, $y\in\mathbb R$ and $t>0$ being parameters.
\begin{RHP}\label{RH-1}
Find an analytic function $M(y,t;z)$ with the following properties:
\begin{itemize}
  \item $M(y,t;z)$ is meromorphic in $\mathbb{C}\setminus\mathbb{R}$ and continuous to real axis;
  \item $M(y,t;-z)=M(y,t;z)\sigma_{1}$ where $\sigma_{1}$ is the standard Pauli matrix which can be expressed as
      \begin{align*}
 \sigma_{1}=\left(
                            \begin{array}{cc}
                              0 & 1 \\
                              1 & 0 \\
                            \end{array}
                          \right),~~\sigma_{2}=\left(
                            \begin{array}{cc}
                              0 & -i \\
                              i & 0 \\
                            \end{array}
                          \right),~~\sigma_{3}=\left(
                            \begin{array}{cc}
                              1 & 0 \\
                              0 & -1 \\
                            \end{array}
                          \right);
\end{align*}
  \item $M_{+}(y,t;z)=M_{-}(y,t;z)V(y,t;z)$,~~~$z\in\mathbb{R}$,
  where \begin{align}\label{J-Matrix}
V(y,t;z)=e^{-iz\left(y-\frac{2t}{1+4z^2}\right)\hat{\sigma}_{3}}\left(\begin{array}{cc}
                   1-|r(z)|^{2} & -\overline{r(z)} \\
                   r(z)& 1
                 \end{array}\right),
\end{align}
where $e^{\hat{\sigma}_{3}}A$ means $e^{\sigma_{3}}Ae^{-\sigma_{3}}$;
  \item $M(y,t;z)\to(1~~1)$ as $z\rightarrow\infty$;
  \item $M(y,t;z)$ possesses simple poles at each points in $\mathcal{Z}\cup\bar{\mathcal{Z}}$ with:
      \begin{align}\label{res-M}
      \begin{split}
\mathop{Res}_{z=z_{j}}M(y,t;z)&=\lim_{z\rightarrow z_{j}}M(y,t;z)\left(\begin{array}{cc}
                   0 & 0 \\
                   i\gamma_{j}e^{2it\theta(z_j)} & 0
                 \end{array}\right),\\
\mathop{Res}_{z=\bar{z}_{j}}M(y,t;z)&=\lim_{z\rightarrow \bar{z}_{j}}M(y,t;z)\left(\begin{array}{cc}
                   0 & -i\gamma_{j}e^{-2it\theta(\bar{z}_j)} \\
                   0 & 0
                 \end{array}\right).
\end{split}
\end{align}
where $\theta(y,t,z)=z\left(\frac{y}{t}-\frac{2}{1+4z^2}\right)$.
\end{itemize}

\end{RHP}

Then, according to the RH problem \ref{RH-1}, the solution $q(x,t)$ of the initial value problem \eqref{CH-equation} can be expressed in a parametric form
\begin{align}
  x(y,t)&=y+\ln\frac{M_1(y,t;\frac{i}{2})}{M_2(y,t;\frac{i}{2})},\label{solution-1}\\
  q(y,t)&=\frac{1}{2i}\lim_{z\to\frac{i}{2}}\left(\frac{M_1(y,t;z)M_2(y,t;z)} {M_1(y,t;\frac{i}{2})M_2(y,t;\frac{i}{2})}-1\right)\frac{1}{z-\frac{i}{2}}. \label{solution-2}
\end{align}
Alternatively,
$$q(y,t)=\frac{\partial}{\partial t}\ln\frac{M_1(y,t;\frac{i}{2})}{M_2(y,t;\frac{i}{2})}.$$
\end{prop}

\begin{rem}
On the basis of the Zhou's vanishing lemma \cite{vanishing-lemma}, the existence of the solutions of RHP \ref{RH-1} for $x\in\mathbb{R},t>0$ is guaranteed. According to the results of Liouville's theorem, we know that if the solution of RH problem \ref{RH-1} exists, it is unique.
\end{rem}

\begin{prop}\label{H-r-space}
If the initial data $q_{0}(x)\in H^{4,2}(\mathbb R)$, then $r(z)\in H^{1,1}(\mathbb{R})$. The space $H^{4,2}(\mathbb R)$ is defined as
\begin{align*}
  H^{k,2}(\mathbb{R})=\left\{f(x)\in L^{p}(\mathbb{R}):x^{2}\partial^{j}f(x)\in L^{p}(\mathbb{R}), j=1,2,\ldots,k\right\}.
\end{align*}
\end{prop}

\section{Conjugation}\label{Conjugation}

According to the jump matrix \eqref{J-Matrix}, we know that  the oscillation term  is
\begin{align}\label{4.1}
e^{2iz\left(y-\frac{2t}{1+4z^2}\right)}=e^{2it\theta(z)},~~ \theta(z)=z\left(\frac{y}{t}-\frac{2}{1+4z^2}\right).
\end{align}

It is obvious to observe that the growth and decay of the exponential function $e^{2it\theta(z)}$ have great effect on the long-time asymptotic of RH problem \ref{RH-1}. Therefore, we introduce a transformation $M(z)\rightarrow M^{(1)}(z)$ to make the $M^{(1)}(z)$ is well behaved as $|t|\rightarrow \infty$ along any characteristic line.

Let $\xi=\frac{y}{t}$, the stationary phase points, i.e., $\theta'(z)=0$, are given by $\pm z_0$ and $\pm z_1$ where
\begin{align*}
  z_0=\frac{1}{2}\sqrt{-\frac{\xi+1-\sqrt{1+4\xi}}{\xi}},\\
  z_1=\frac{1}{2}\sqrt{-\frac{\xi+1+\sqrt{1+4\xi}}{\xi}}.
\end{align*}
In order to study the asymptotic behavior of $e^{2it\theta(z)}$ as $t\to\infty$, we evaluate the real part of $2it\theta(z)$:
\begin{align}\label{real-theta}  Rez(i\theta)=-Imz\xi-\frac{16Rez^2Imz-2Imz(1+4Rez^2-4Imz^2)}{(1+4Rez^2-4Imz^2)^2 +(8RezImz)^2}.
\end{align}
The properties of $Re(i\theta)$ are shown in Fig. \ref{fig-1}

\begin{figure}
 \centering
  \includegraphics[width=6.6cm,height=6.2cm,angle=0]{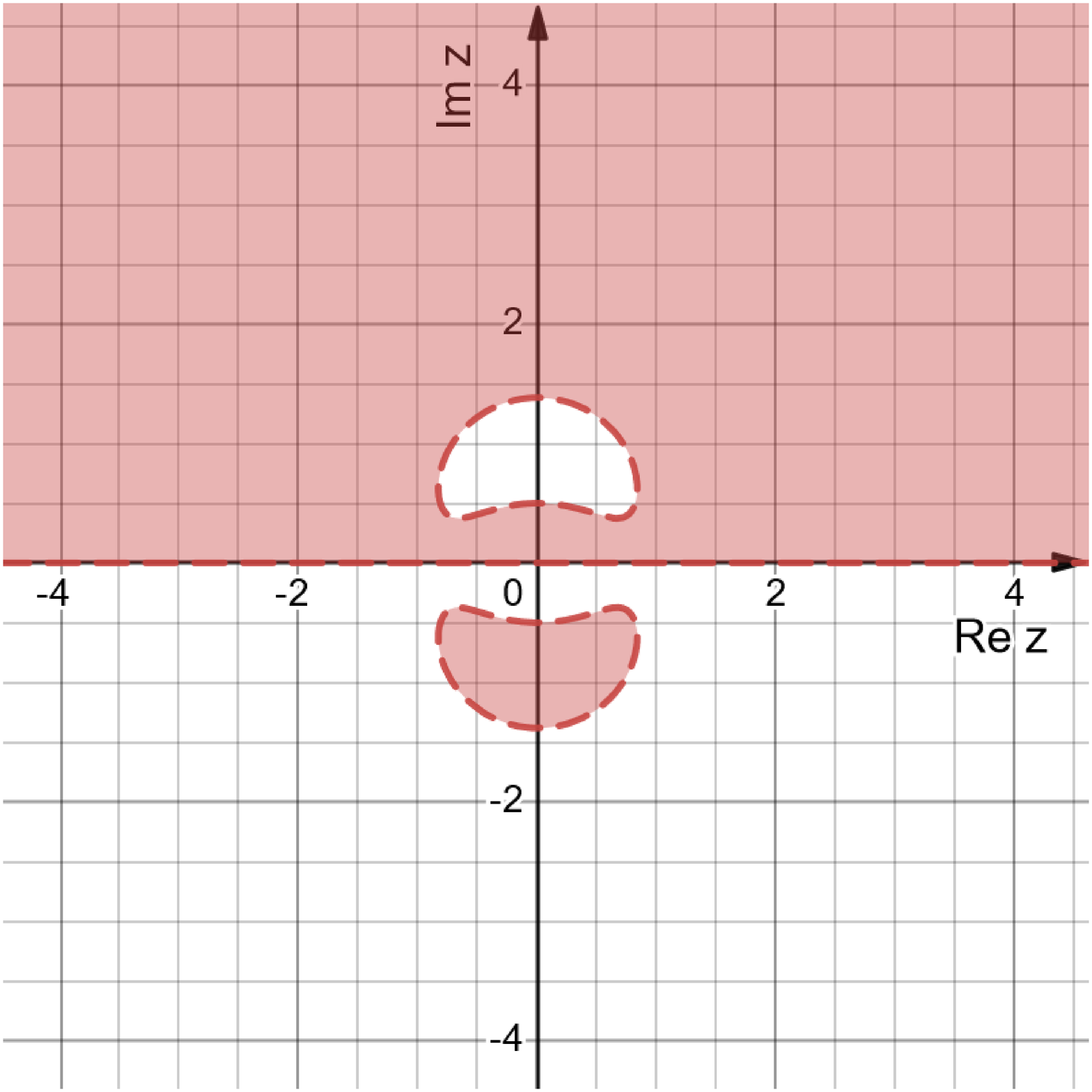}
  \includegraphics[width=6.6cm,height=6.2cm,angle=0]{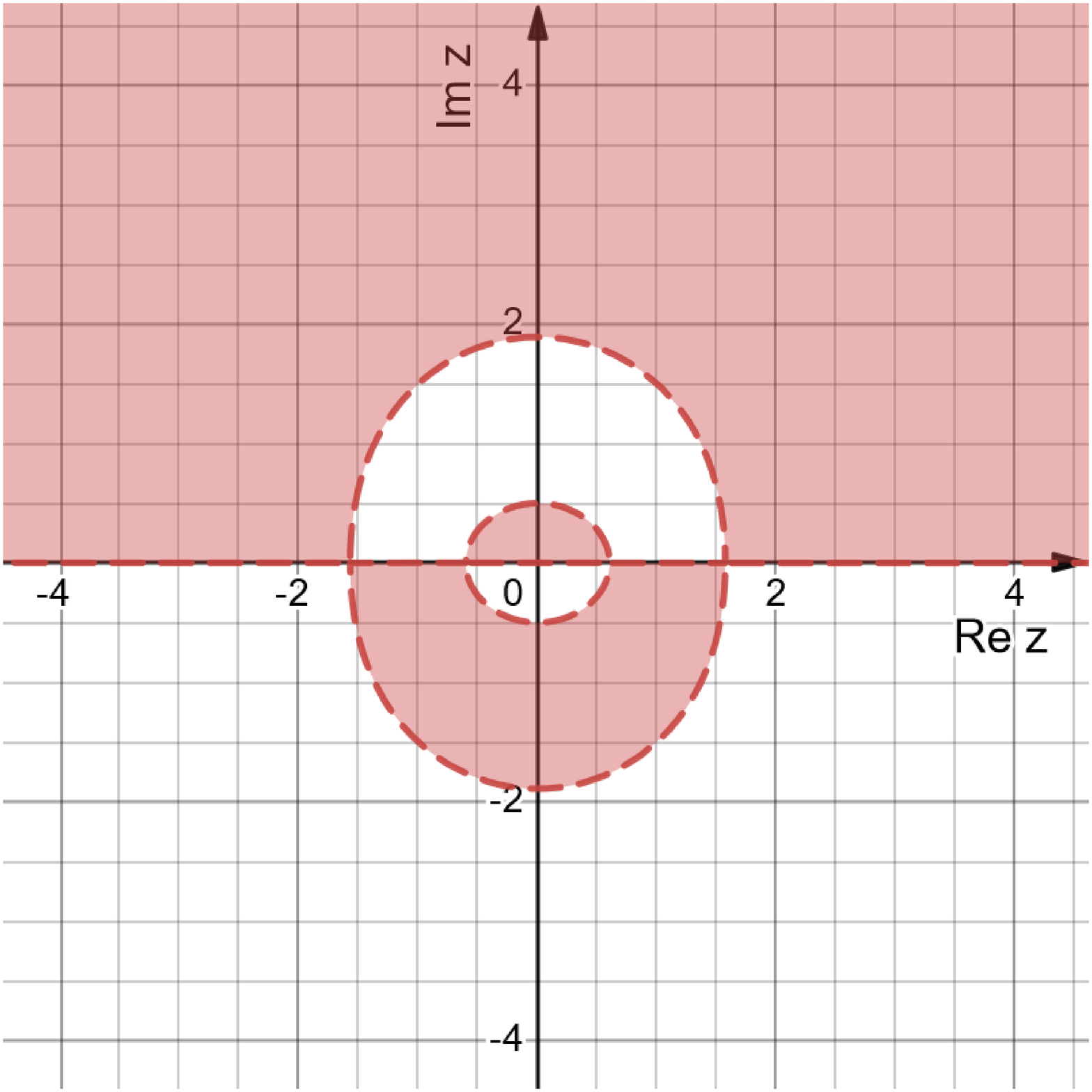}

  (\textbf{$a$})\qquad  \qquad\qquad\qquad\qquad\qquad\qquad\qquad\qquad~~~(\textbf{$b$})\\

  \includegraphics[width=6.6cm,height=6.2cm,angle=0]{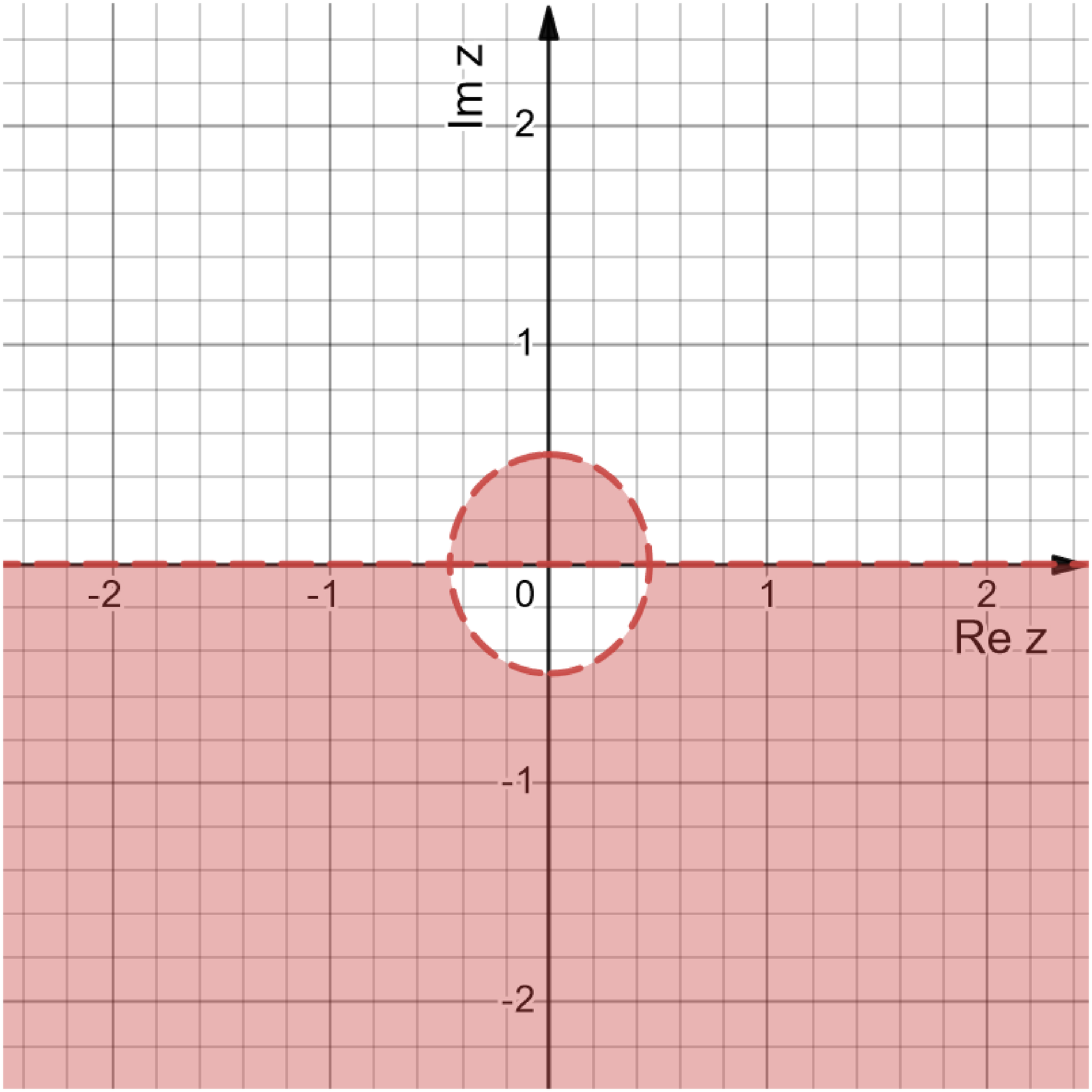}
  \includegraphics[width=6.6cm,height=6.2cm,angle=0]{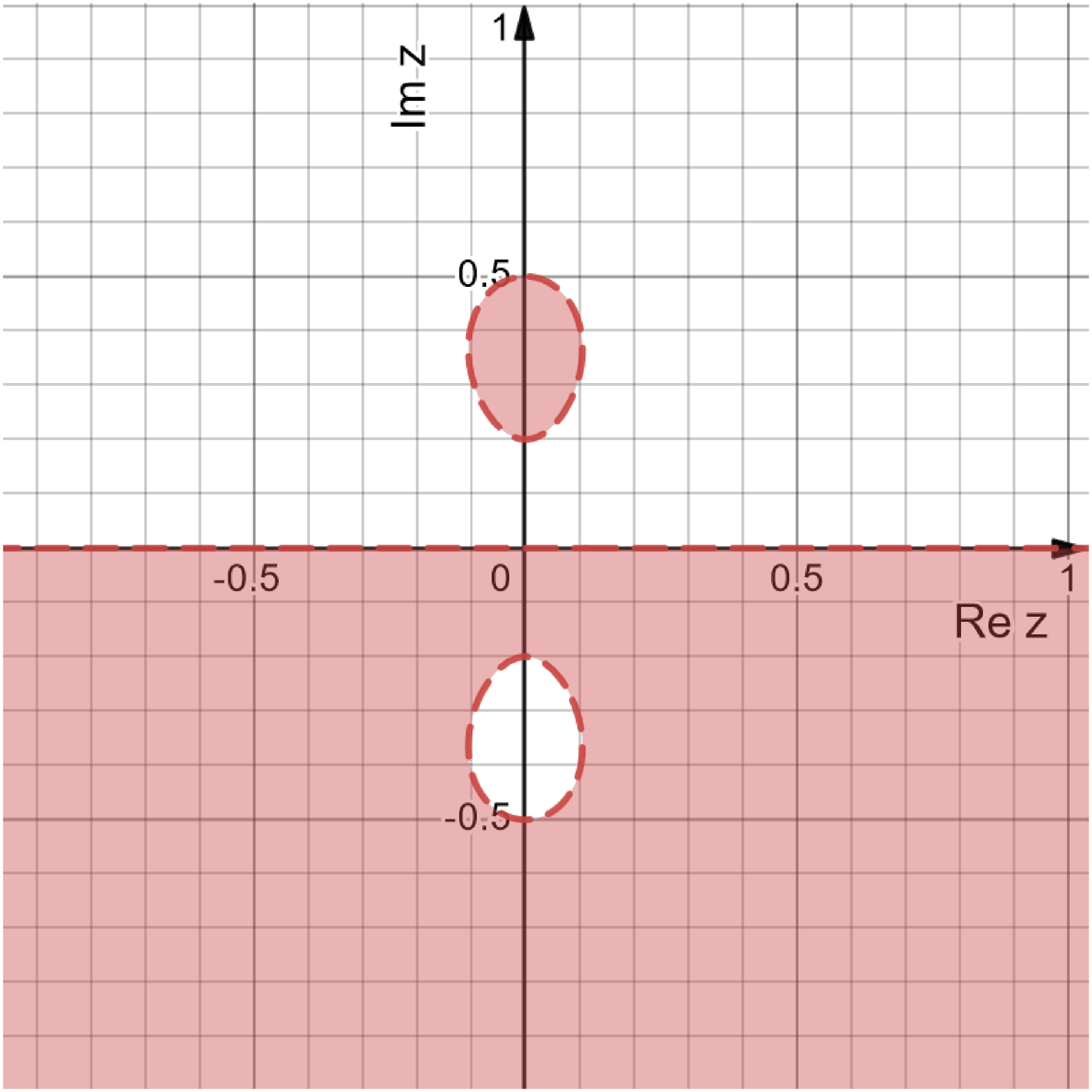}

  (\textbf{$c$})\qquad  \qquad\qquad\qquad\qquad\qquad\qquad\qquad\qquad~~~(\textbf{$d$})
%\qquad\qquad\qquad\qquad(\textbf{f})

  \caption{\small (Color online) The classification of sign $Re(2i\theta(z))$.  In the pink regions, $Re(2i\theta(z))>0$, which
implies that $e^{-2it\theta(z)}\to0$ as $t\to\infty$. While in the white regions, $Re(2i\theta(z))<0$, which implies $e^{2it\theta(z)}\to0$ as $t\to\infty$. The red curves $Re(2i\theta(z))=0$ are critical lines between decay and growth regions.}\label{fig-1}
\end{figure}

Therefore, there are four cases are distinguished as
\begin{enumerate}[(i)]
  \item For $\xi>2$, $z_0,z_1\notin\mathbb R$ corresponding to Fig. \ref{fig-1}($d$);
  \item For $0<\xi<2$, $z_0\in \mathbb R, z_1\notin\mathbb R$ corresponding to Fig. \ref{fig-1}($c$);
  \item For $-\frac{1}{4}<\xi<0$, $z_0,z_1\in\mathbb R$ corresponding to Fig. \ref{fig-1}($b$);
  \item For $\xi<-\frac{1}{4}$, $z_0,z_1\notin\mathbb R$ corresponding to Fig. \ref{fig-1}($a$).
\end{enumerate}

Next, for convenience, we introduce some notations.
Let $\xi_1=z_1$, $\xi_2=z_0$, $\xi_3=-z_0$, $\xi_4=-z_1$.
Define
\begin{align*}
 I(\xi)=\left\{\begin{aligned}
&\mathbb R, &\xi\in(-\infty,-\frac{1}{4}),\\
&(-\infty,\xi_4)\cup(\xi_3,\xi_2)\cup(\xi_4,\infty),&\xi\in(-\frac{1}{4},0),\\
&(\xi_3,\xi_2),&\xi\in(0,2),\\
&\emptyset,&\xi\in( 2,\infty).
\end{aligned}\right.
\end{align*}
Denote
\begin{align*}
 n(\xi)=\left\{\begin{aligned}
&0, &\xi\in(-\infty,-\frac{1}{4})\cup( 2,\infty),\\
&2,&\xi\in(-\frac{1}{4},0),\\
&4,&\xi\in(0,2).
\end{aligned}\right.
\end{align*}
Define
\begin{align}
 \rho=\frac{1}{3}\min_{z_j\in\mathcal{Z},j\neq k}\{z_j-z_k\},\label{rho}\\
 \mathcal{N}=\{1,2,\ldots,N\},\label{N-1-N}
\end{align}
\begin{align}\label{j-0}
j_0=j_0(\xi)=\left\{\begin{aligned}
&j,~~ if ~|\frac{1}{2}-Imz_j|<\rho,~~or~~|\kappa_0-Imz_j|<\rho, ~j\in\{1,2,\ldots,N\},~~for~~ \xi\in(2,\infty),\\
&j,~~ if ~|\frac{1}{2}-Imz_j|<\rho, ~j\in\{1,2,\ldots,N\},~~for~~ \xi\in(-\infty, 2),\\
&-1,~~ otherwise.
\end{aligned}\right.
\end{align}

It is noted that the discrete spectrum points $z_j(j=1,2,\ldots,N)$ are distribute in $i(0,\frac{1}{2})$, which means the discrete spectrum points $z_j(j=1,2,\ldots,N)$ are distribute in the positive region $Re(2i\theta(z))>0$ for $\xi<2$. For $\xi>2$, a part of discrete spectrum points is distribute in the positive region $Re(2i\theta(z))>0$ and another part is distribute in the negative region $Re(2i\theta(z))<0$.

Therefore, we further  define functions
\begin{align}\label{delta-v-define}
\delta(z):&=\delta(z,\xi)=\exp\left(i\int_{I(\xi)}\frac{\nu(s)}{s-z}\,ds\right), ~~\nu(s)=-\frac{1}{2\pi}\log(1+|r(s)|^{2}),
\end{align}
\begin{align}\label{T-define}
T(z):&=T(z,\xi)=
\left\{\begin{aligned}
&\prod_{k=1}^{N}\frac{(z-\bar{z}_{k})} {(z-z_{k})}\delta(z),&\xi<2,\\
&\prod_{\kappa_0<Imz_k<\frac{1}{2}}\frac{(z-\bar{z}_{k})} {(z-z_{k})}\delta(z), &\xi>2.
\end{aligned}\right.
\end{align}
where $\kappa_0=\frac{1}{2}\sqrt{1-\frac{2}{\xi}}$.

Then, we show   the properties of the function $T(z)$.

\begin{prop}\label{T-property}
The functions $T(z)$ satisfies that\\
($a$) $T(z)$  is meromorphic in $C\setminus\mathbb{R}$; \\
($b$) For $z\in C\setminus\mathbb{R}$, $\bar{T}(\bar{z})=T(-z)$;\\
($c$) For $z\in C\setminus I(\xi)$, $T^{-1}(z)=T(-z)$;\\
($d$) For $z\in I(\xi)$, the boundary values $T_{\pm}$  satisfies that
\begin{align}\label{T-jump}
\frac{T_{+}(z)}{T_{-}(z)}=1-|r(z)|^{2}, z\in I(\xi);
\end{align}
($e$) $T(\infty):=\lim\limits_{z\rightarrow\infty}T(z)=1$;\\
($f$) As $z\rightarrow\frac{i}{2}$, $T(z)$ can be  expressed as
\begin{align}\label{4.10}
%\begin{split}
T(z)=T(\frac{i}{2})(\mathbb I+(z-\frac{i}{2})T_{1})+O((z-\frac{i}{2})^{2}),
%\end{split}
\end{align}
where $T_{1}=2i\sum\frac{2iIm~z_{j}}{\frac{1}{4}-Imz^2_{j}} -\frac{1}{2\pi i}\int_{I(\xi)}\frac{\log(1-|r(s)|^2}{(s-z)^{2}}\,ds$;\\
($g$) As $z\rightarrow \xi_j(j=1,2,3,4)$ along any ray $\xi_j+e^{i\phi}R_{+}$ with $|\phi|\leq c<\pi$,
\begin{align}\label{T-estimate}
  \left|T(z,\xi)-T_j(\xi)(z-\xi_j)^{i\nu(\xi_j)}\right| \lesssim||r||_{H^{1,1}(\mathbb R)}|z-\xi_j|^{\frac{1}{2}},
\end{align}
where
\begin{align}
T_j(\xi_j)=\prod\frac{\xi_j+iImz_j}{\xi_j-iImz_j}e^{i\beta(\xi_j,\xi_j)},
\end{align}
with
\begin{align}\label{3-6}
\beta(z,\xi_j)=\nu(\xi_j)\log(z-\xi_j+1)+\int_{I(\xi)}\frac{\nu(s)}{s-z}\,ds.
\end{align}
\end{prop}

\begin{proof}
The above properties of $T(z)$ can be proved by a direct calculation, for details, see \cite{AIHP,Li-cgNLS}.
\end{proof}

\section{Deformation of Riemann-Hilbert problem}\label{Deformation of RH problem}

In order to trade the poles for jumps on small contours encircling the pole $z_j(j\neq j_0)$ (see Fig. \ref{Fig-V-1}) and deform   the Riemann-Hilbert problem \ref{RH-1}, through employing the function $T(z)$, we introduce the following transformation: for $\xi\in(-\infty,2)$ and $j\neq j_0$,
\begin{align}\label{Trans-1}
M^{(1)}(y,t,z)=\left\{\begin{aligned}
      &M(y,t,z)\left(
        \begin{array}{cc}
      1 & -\frac{z-iImz_{j}}{i\gamma_{j}e^{2it\theta(iImz_{j})}} \\
      0 & 1 \\
        \end{array}
      \right)T(z)^{-\sigma_{3}}, ~~&|z-iImz_{j}|<\rho,\\
   &M(y,t,z)\left(
        \begin{array}{cc}
      1 & 0 \\
      \frac{z+iImz_{j}}{i\gamma_{j}e^{-2it\theta(-iImz_{j})}} & 1 \\
        \end{array}
      \right)T(z)^{-\sigma_{3}}, ~~&|z+iImz_{j}|<\rho,\\
   &M(y,t,z)T(z)^{-\sigma_{3}}, \qquad &elsewhere,
   \end{aligned}\right.
\end{align}
for $\xi\in(2,\infty)$  and $j\neq j_0$,
\begin{align}\label{Trans-1+1}
M^{(1)}(y,t,z)=\left\{\begin{aligned}
      &M(y,t,z)\left(
        \begin{array}{cc}
      1 & -\frac{z-iImz_{j}}{i\gamma_{j}e^{2it\theta(iImz_{j})}} \\
      0 & 1 \\
        \end{array}
      \right)T(z)^{-\sigma_{3}}, ~~&|z-iImz_{j}|<\rho,~~\kappa_0<Imz_j<\frac{1}{2},\\
   &M(y,t,z)\left(
        \begin{array}{cc}
      1 & 0 \\
      -\frac{-i\gamma_{j}e^{2it\theta(iImz_{j})}}{z-iImz_{j}} & 1 \\
        \end{array}
      \right)T(z)^{-\sigma_{3}}, ~~&|z-iImz_{j}|<\rho,~~0<Imz_j<\kappa_0,\\
   &M(y,t,z)\left(
        \begin{array}{cc}
      1 & 0 \\
      \frac{z+iImz_{j}}{i\gamma_{j}e^{-2it\theta(-iImz_{j})}} & 1 \\
        \end{array}
      \right)T(z)^{-\sigma_{3}}, ~~&|z+iImz_{j}|<\rho,~~\kappa_0<Imz_j<\frac{1}{2},\\
      &M(y,t,z)\left(
        \begin{array}{cc}
      1 & \frac{i\gamma_{j}e^{-2it\theta(-iImz_{j})}}{z+iImz_{j}} \\
      0 & 1 \\
        \end{array}
      \right)T(z)^{-\sigma_{3}}, ~~&|z+iImz_{j}|<\rho,~~0<Imz_j<\kappa_0,\\
   &M(y,t,z)T(z)^{-\sigma_{3}}, \qquad &elsewhere.
   \end{aligned}\right.
\end{align}

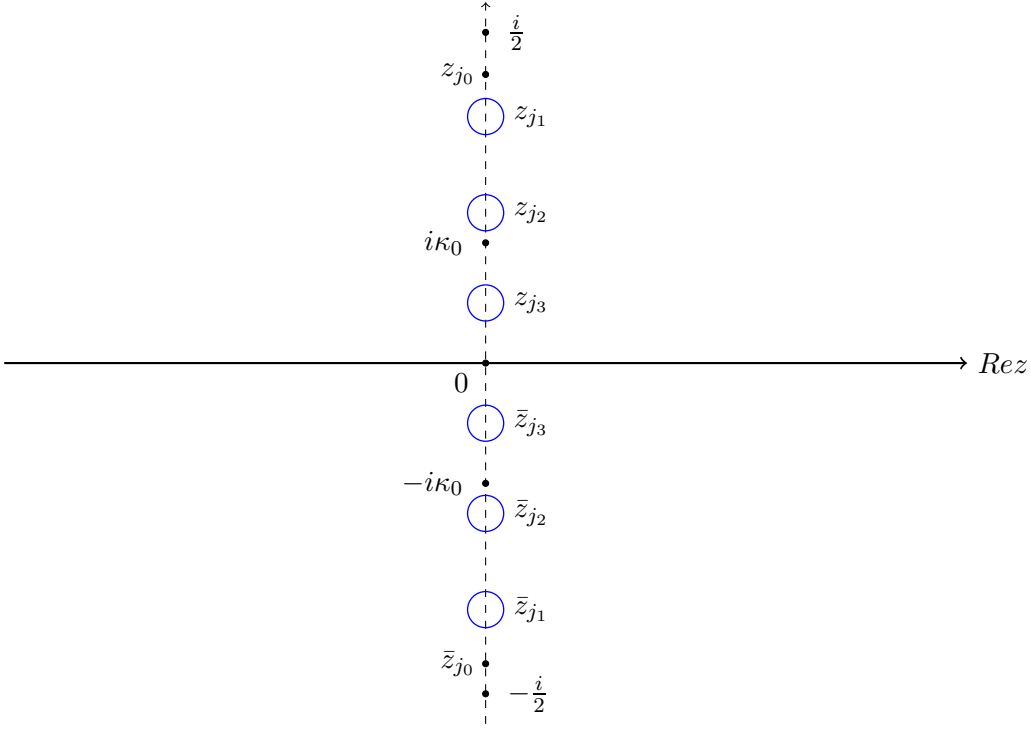
\begin{figure}
\centerline{\begin{tikzpicture}[scale=0.8]
%\path [fill=pink] (-5,4)--(-5,0) to (-9,0) -- (-9,4);
%\path [fill=pink] (-5,-4)--(-5,0) to (-1,0) -- (-1,-4);
\draw[->][thick](-8,0)--(8,0)node[right]{$Rez$};
\draw[fill] (0,0) circle [radius=0.05];
\draw[fill] (-0.4,0)node[below]{$0$} ;
%\draw[fill] (4,0)circle [radius=0.05];
%\draw[fill] (4.2,0)node[below]{$\xi_{0}$};
%\draw[fill] (5.5,0)node[below]{$\xi_{0}+\rho$};
%\draw[fill] (2.5,0)node[below]{$\xi_{0}-\rho$};
%\draw[-][dashed](3.5,5)--(3.5,-5);
%\draw[-][dashed](4.5,5)--(4.5,-5);
\draw[->][dashed](0,-6)--(0,6);
\draw[fill] (0,5.5) circle [radius=0.05];
\draw[fill] (0.2,5.5)node[right]{$\frac{i}{2}$};
\draw[fill] (0,2) circle [radius=0.05];
\draw[fill] (-0.2,2)node[left]{$i\kappa_0$};
\draw[fill] (0,-5.5) circle [radius=0.05];
\draw[fill] (0.2,-5.5)node[right]{$-\frac{i}{2}$};
\draw[fill] (0,-2) circle [radius=0.05];
\draw[fill] (-0.2,-2)node[left]{$-i\kappa_0$};
\draw[fill] (0,4.8) circle [radius=0.05]node[left]{$z_{j_0}$};
\draw(0,4.1) [blue, line width=0.5] circle(0.3);
\draw[fill] (0.3,4.1)node[right]{$z_{j_1}$};
\draw(0,2.5) [blue, line width=0.5] circle(0.3);
\draw[fill] (0.3,2.5)node[right]{$z_{j_2}$};
\draw(0,1) [blue, line width=0.5] circle(0.3);
\draw[fill] (0.3,1)node[right]{$z_{j_3}$};
\draw[fill] (0,-5) circle [radius=0.05]node[left]{$\bar{z}_{j_0}$};
\draw(0,-4.1) [blue, line width=0.5] circle(0.3);
\draw[fill] (0.3,-4.1)node[right]{$\bar{z}_{j_1}$};
\draw(0,-2.5) [blue, line width=0.5] circle(0.3);
\draw[fill] (0.3,-2.5)node[right]{$\bar{z}_{j_2}$};
\draw(0,-1) [blue, line width=0.5] circle(0.3);
\draw[fill] (0.3,-1)node[right]{$\bar{z}_{j_3}$};
\end{tikzpicture}}
\caption{\small (Color online) The contour of the regular RH problem.}\label{Fig-V-1}
\end{figure}

Then, we obtain a matrix RH problem for $M^{(1)}(z)$.
\begin{RHP}\label{RH-2}
Find a matrix function $M^{(1)}(y,t,z)$ with the following properties:
\begin{itemize}
  \item $M^{(1)}(y,t,z)$ is meromorphic in $\mathbb{C}\setminus \Sigma^{(\xi,1)}$,
  where\begin{align*}
  \Sigma^{(\xi,1)}=\mathbb R\cup\left(\bigcup_{j\in\{1,2,\ldots,N\},j\neq j_0} \{z\in\mathbb C:|z-iImz_j|=\rho~or~|z+iImz_j|=\rho\}\right);
       \end{align*}
  \item $M^{(1)}(y,t;-z)=M^{(1)}(y,t;z)\sigma_{1}$;
  \item $M^{(1)}(y,t;z)\to(1~~1)$ as $z\rightarrow \infty$;
  \item For $z\in \Sigma^{(\xi,1)} $, the boundary values $M^{(1)}_{\pm}(y,t,z)$ satisfy the jump relationship $M^{(1)}_{+}(y,t,z)=M^{(1)}_{-}(y,t,z)V^{(1)}(z)$, where for $\xi\in(-\infty,2)$,
      \begin{align}\label{RH-2-V1}
       V^{(1)}(z)=\left\{\begin{aligned}
       &\left(
    \begin{array}{cc}
    1 &  -\overline{r(z)}T(z)^{2}e^{-2it\theta(z)} \\
    0 & 1 \\
     \end{array}
   \right)\left(
    \begin{array}{cc}
    1 & 0 \\
    r(z)T(z)^{-2}e^{2it\theta(z)} & 1 \\
   \end{array}
  \right),~~z\in\mathbb R \setminus I(\xi)\\
   &\left(
    \begin{array}{cc}
    1 & 0 \\
    \frac{r(z)T_{-}(z)^{-2}}{1-|r(z)|^2}e^{2it\theta(z)} & 1 \\
     \end{array}
   \right)\left(
    \begin{array}{cc}
    1 & \frac{-\overline{r(z)}T_{+}(z)^{2}}{1-|r(z)|^2}e^{-2it\theta(z)} \\
    0 & 1 \\
   \end{array}
  \right),~~z\in I(\xi),\\
   &\left(
        \begin{array}{cc}
      1 & -\frac{z-iImz_{j}}{i\gamma_{j}e^{2it\theta(iImz_{j})}}T(z)^{2} \\
      0 & 1 \\
        \end{array}
      \right), ~~|z-iImz_{j}|= \rho,~ j\neq j_0,\\
   &\left(
        \begin{array}{cc}
      1 & 0 \\
      -\frac{z+iImz_{j}}{i\gamma_{j}e^{-2it\theta(-iImz_{j})}}T(z)^{-2} & 1 \\
        \end{array}
      \right),~~|z+iImz_{j}|= \rho,~ j\neq j_0,
   \end{aligned}\right.
  \end{align}
  for $\xi\in(2,\infty)$,
  \begin{align}\label{RH-2-V1+1}
       V^{(1)}(z)=\left\{\begin{aligned}
       &\left(
    \begin{array}{cc}
    1 &  -\overline{r(z)}T(z)^{2}e^{-2it\theta(z)} \\
    0 & 1 \\
     \end{array}
   \right)\left(
    \begin{array}{cc}
    1 & 0 \\
    r(z)T(z)^{-2}e^{2it\theta(z)} & 1 \\
   \end{array}
  \right),~~z\in\mathbb R\\
   &\left(
        \begin{array}{cc}
      1 & -\frac{z-iImz_{j}}{i\gamma_{j}e^{2it\theta(iImz_{j})}}T(z)^{2} \\
      0 & 1 \\
        \end{array}
      \right), ~~|z-iImz_{j}|= \rho,~~\kappa_0<Imz_j<\frac{1}{2},~ j\neq j_0,\\
    &\left(
        \begin{array}{cc}
      1 & 0 \\
      -\frac{i\gamma_{j}e^{2it\theta(iImz_{j})}}{z-iImz_{j}}T(z)^{-2} & 1 \\
        \end{array}
      \right), ~~|z-iImz_{j}|= \rho,~~0<Imz_j<\kappa_0,~ j\neq j_0,\\
   &\left(
        \begin{array}{cc}
      1 & 0 \\
      -\frac{z+iImz_{j}}{i\gamma_{j}e^{-2it\theta(-iImz_{j})}}T(z)^{-2} & 1 \\
        \end{array}
      \right), ~~|z+iImz_{j}|= \rho,~~\kappa_0<Imz_j<\frac{1}{2},~ j\neq j_0,\\
      &\left(
        \begin{array}{cc}
      1 & -\frac{i\gamma_{j}e^{-2it\theta(-iImz_{j})}}{z+iImz_{j}}T(z)^{2} \\
      0 & 1 \\
        \end{array}
      \right), ~~|z+iImz_{j}|= \rho,~~0<Imz_j<\kappa_0,~ j\neq j_0;
   \end{aligned}\right.
  \end{align}
   \item  $M^{(1)}(y,t,z)$ has simple poles at each $z_{j}\in \mathcal{Z}(j=j_0)$ at which for $\xi\in(-\infty,2)$,
\begin{align}\label{ResM1}
\begin{split}
\mathop{Res}\limits_{z=z_{j_0}}M^{(1)}(y,t,z)=\lim_{z\rightarrow z_{j_0}}M^{(1)}(y,t,z)\left(\begin{array}{cc}
    0 & i\gamma_{j_0}e^{-2it\theta(z_{j_0})}\frac{1}{T}'(z_{j_0})^{-2} \\
    0 & 0 \\
  \end{array}
\right),\\
\mathop{Res}\limits_{z=\bar{z}_{j_0}}M^{(1)}(y,t,z)=\lim_{z\rightarrow \bar{z}_{j_0}}M^{(1)}(y,t,z)\left(\begin{array}{cc}
    0 & 0 \\
    i\gamma^{-1}_{j_0}e^{2it\theta(\bar{z}_{j_0})}T'(\bar{z}_{j_0})^{-2} & 0 \\
  \end{array}
\right);
\end{split}
\end{align}
For   $\xi\in(2,\infty)$, if $\kappa_0<Imz_{j_0}<\frac{1}{2}$, the residue conditions are the same with \eqref{ResM1}; if $0<Imz_{j_0}<\kappa_0$, the residue conditions are that
\begin{align}\label{ResM1+1}
\begin{split}
\mathop{Res}\limits_{z=z_{j_0}}M^{(1)}(y,t,z)=\lim_{z\rightarrow z_{j_0}}M^{(1)}(y,t,z)\left(\begin{array}{cc}
    0 & 0 \\
    i\gamma_{j_0}e^{2it\theta(z_{j_0})}T^{-2}(z_{j_0}) & 0 \\
  \end{array}
\right),\\
\mathop{Res}\limits_{z=\bar{z}_{j_0}}M^{(1)}(y,t,z)=\lim_{z\rightarrow \bar{z}_{j_0}}M^{(1)}(y,t,z)\left(\begin{array}{cc}
    0 & i\gamma_{j_0}e^{-2it\theta(\bar{z}_{j_0})}T^{2}(\bar{z}_{j_0}) \\
    0 & 0 \\
  \end{array}
\right).
\end{split}
\end{align}
\end{itemize}
\end{RHP}

\begin{proof}
According to the definition of $M^{(1)}(y,t,z)$, applying Proposition \ref{T-property}, the analyticity, jump matrix and asymptotic behavior of $M^{(1)}$ are derived directly. Moreover, the residue condition of $M^{(1)}(y,t,z)$ can be derived by a similar calculation.
\end{proof}

Then, according to the RH problem \ref{RH-2}, the solution $q(x,t)$ of the initial value problem \eqref{CH-equation} can be expressed in a parametric form
\begin{align}\label{1-solution-1}
\begin{split}
  x(y,t)&=y+\ln\frac{M^{(1)}_1(y,t;\frac{i}{2})} {M^{(1)}_2(y,t;\frac{i}{2})},\\
  q(y,t)&=\frac{1}{2i}\lim_{z\to\frac{1}{2}}\left(\frac{M^{(1)}_1(y,t;z) M^{(1)}_2(y,t;z)}{M^{(1)}_1(y,t;\frac{i}{2})M^{(1)}_2(y,t;\frac{i}{2})}-1\right) \frac{1}{z-\frac{i}{2}}.
\end{split}
\end{align}

\section{The $\bar{\partial}$ extensions of jump factorization}\label{Opening-lenses}
The goal of this section  is to  removing the jump from the real axis $\mathbb R$ in new lines along which the $e^{2it\theta(z)}$ is decay/growth for $z\notin \mathbb{R}$. To approach this goal, we first introduce some regions and contours with respect to $\xi$.

Fix a angle $\theta_0$ (small enough)  such that the set $\{z\in\mathbb{C}: ~\cos\theta_0<\big|\frac{Rez}{z}\big|\}$ does not intersect the boundary $Re(i\theta(z))=0$.
Define
\begin{align*}
\phi=\min\{\theta_{0},~\frac{\pi}{4}\}.
\end{align*}
\begin{itemize}
  \item For $\xi\in(-\infty,-\frac{1}{4})\cup(2,\infty)$, following the idea in \cite{CH-Longtime,three-wave-Yang}, let $\varepsilon_0$ be an sufficiently small positive constant such that the circles $\{|z-z_j|=\rho, |z-\bar{z}_j|=\rho\}$ around $z_j$ and $\bar{z}_j$ lie outside the region $\{z||Imz|<\varepsilon_0\}$.
  Then, we define
  $\Omega(\xi)=\bigcup_{k=1}^{4}\Omega_{k}$ where $\Omega_{k}=\Omega_{k0}\cup\Omega_{k1}$ (see Fig. \ref{Fig-2})  with
\begin{align*}
&\Omega_{10}=\{z: Rez>0,0<Imz<\varepsilon_0\},~~\Omega_{11}=\{z: \arg(z-i\varepsilon_0)\in(0,\phi)\}, \\
&\Omega_{20}=\{z: Rez<0,0<Imz<\varepsilon_0\},~~\Omega_{21}=\{z: \arg(z-i\varepsilon_0)\in(\pi-\phi,\pi)\},\\
&\Omega_{30}=\{z: Rez>0,-\varepsilon_0<Imz<0\},~~\Omega_{31}=\{z: \arg(z+i\varepsilon_0)\in(-\pi,-\pi+\phi)\}, \\
&\Omega_{40}=\{z: Rez<0,-\varepsilon_0<Imz<0\},~~\Omega_{41}=\{z: \arg(z+i\varepsilon_0)\in(-\phi,0)\}.
\end{align*}
Denote
\begin{align}\label{S-Sigma}
\begin{split}
\Sigma_{1}=i\varepsilon_0+e^{i\phi}\mathbb{R}_{+},~~~~~~\Sigma_{2}=i\varepsilon_0+e^{i(\pi-\phi)}\mathbb{R}_{+},\\
\Sigma_{3}=-i\varepsilon_0+e^{-i(\pi-\phi)}\mathbb{R}_{+},~~\Sigma_{4}=-i\varepsilon_0+e^{-i\phi}\mathbb{R}_{+}.
\end{split}
\end{align}
  Define $$\tilde{\Sigma}(\xi)=\Sigma_{1}\cup\Sigma_{2}\cup\Sigma_{3}\cup\Sigma_{4},$$
  $$\Sigma^{(\xi,2)}=\tilde{\Sigma}(\xi)\cup\left(\bigcup_{j\in\mathcal{N},j\neq j_0}\{|z-z_j|=\rho,|z-\bar z_j|=\rho\}\right),$$
  where $\mathcal{N}=\{1,2,\ldots,N\}$.
  \item  For $\xi\in(-\frac{1}{4}, 0)$ (see Fig.\ref{case1}), we define
      \begin{align*}
        \Sigma_{jk}(\xi)=\left\{\begin{aligned}
        &\xi_j+e^{i(\pi+(-1)^k\phi)},~~k=1,2,\\
        &\xi_j+e^{i(2\pi+(-1)^k\phi)},~~k=3,4,
        \end{aligned}\right.
      \end{align*}
      where $j=1,3,$ and for $j=2,3,$
      \begin{align*}
        \Sigma_{jk}(\xi)=\left\{\begin{aligned}
        &\xi_j+e^{i(-1)^{k-1}\phi},~~k=1,2,\\
        &\xi_j+e^{i(\pi+(-1)^{k-1}\phi)},~~k=3,4,
        \end{aligned}\right.
      \end{align*}
      For $j=2,3,4$, we define
      \begin{align*}
        \Sigma'_{j,\pm}=\frac{\xi_j+\xi{j-1}}{2}\pm e^{i\frac{\pi}{2}}\ell_j,
      \end{align*}
      where $\ell_j\in(0,\frac{|\xi_{j}-\xi_{j-1}|}{2}\tan\phi).$
      Then, we can further define
      \begin{align*}
        \Omega(\xi)&=\bigcup_{j,k=1,2,3,4}\Omega_{jk},~~\Omega_{\pm}=\mathbb C\setminus\Omega(\xi),\\
        \tilde{\Sigma}(\xi)&=\left(\bigcup_{j,k=1,2,3,4}\Sigma_{jk}\right)\cup \left(\bigcup_{j=2,3,4}\Sigma'_{j,\pm}\right).\\
        \Sigma^{(\xi,2)}&=\tilde{\Sigma}(\xi)\cup\left(\bigcup_{j\in\mathcal{N},j\neq j_0}\{|z-z_j|=\rho,|z-\bar z_j|=\rho\}\right).
      \end{align*}
  \item  For $\xi\in(0,2)$ (see Fig.\ref{case2}), we define
  \begin{align*}
        \Sigma_{1k}(\xi)&=\left\{\begin{aligned}
        &\xi_2+e^{i(-1)^{k-1}\phi},~~k=1,2,\\
        &\xi_2+e^{i(\pi+(-1)^{k-1}\phi)},~~k=3,4,
        \end{aligned}\right.\\
        \Sigma_{2k}(\xi)&=\left\{\begin{aligned}
        &\xi_3+e^{i(\pi+(-1)^{k}\phi},~~k=1,2,\\
        &\xi_3+e^{i(-1)^{k-1}\phi},~~k=3,4,
        \end{aligned}\right.\\
        \Sigma'_{3,\pm}(\xi)&=\pm e^{i\frac{\pi}{2}}\ell_3,
      \end{align*}
  where $\ell_3\in(0,\frac{|\xi_{3}-\xi_{2}|}{2}\tan\phi).$
  Then, we can further define
      \begin{align*}
        \Omega(\xi)&=\bigcup_{j=1,2,k=1,2}\Omega_{jk},~~\Omega_{\pm}=\mathbb C\setminus\Omega(\xi),\\
        \tilde{\Sigma}(\xi)&=\left(\bigcup_{j=1,2,k=1,2}\Sigma_{jk}\right)\cup \Sigma'_{3,\pm}.\\
        \Sigma^{(\xi,2)}&=\tilde{\Sigma}(\xi)\cup\left(\bigcup_{j\in\mathcal{N},j\neq j_0}\{|z-z_j|=\rho,|z-\bar z_j|=\rho\}\right).
      \end{align*}
\end{itemize}

\begin{figure}
\centering
\begin{tikzpicture}[node distance=5cm]
		\draw[pink, fill=pink!40] (0,-0.5)--(3,-1)--(3,1)--(0,0.5)--(0,-0.5)--(-3,-1)--(-3,1)--(0,0.5);
		\draw[thick](0,0.5)--(3,1)node[above]{$\Sigma_1$};
		\draw[thick](0,0.5)--(-3,1)node[left]{$\Sigma_2$};
		\draw[thick](0,-0.5)--(-3,-1)node[left]{$\Sigma_3$};
		\draw[thick](0,-0.5)--(3,-1)node[right]{$\Sigma_4$};
		\draw[->](-4,0)--(4,0)node[right]{ Re$z$};
		\draw[->](0,-2.7)--(0,3)node[above]{ Im$z$};
		\draw[-][thick,red](-3,0.5)--(3,0.5);
		\draw[-][thick,red](-3,-0.5)--(3,-0.5);
		\draw[-latex][thick](0,-0.5)--(-1.5,-0.75);
		\draw[-latex][thick](0,0.5)--(-1.5,0.75);
		\draw[-latex][thick](0,0.5)--(1.5,0.75);
		\draw[-latex][thick](0,-0.5)--(1.5,-0.75);
		\coordinate (I) at (0.2,0);
		\coordinate (C) at (-0.2,2.2);
		\coordinate (D) at (2.2,0.2);
		\fill (D) circle (0pt) node[right] {\small$\Omega_{11}$};
		\coordinate (J) at (-2.2,-0.2);
		\fill (J) circle (0pt) node[left] {\small$\Omega_{31}$};
		\coordinate (k) at (-2.2,0.2);
		\fill (k) circle (0pt) node[left] {\small$\Omega_{21}$};
		\coordinate (k) at (2.2,-0.2);
		\fill (k) circle (0pt) node[right] {\small$\Omega_{41}$};
		\fill (I) circle (0pt) node[below] {$0$};
		\coordinate (a) at (2.2,0.7);
		\fill (a) circle (0pt) node[right] {\small$\Omega_{10}$};
		%\coordinate (b) at (0.2,2);
%		\fill (b) circle (0pt) node[right] {\small$\Omega_{2}$};
		\coordinate (c) at (-2.2,0.7);
		\fill (c) circle (0pt) node[left] {\small$\Omega_{20}$};
		\coordinate (d) at (-2.2,-0.7);
		\fill (d) circle (0pt) node[left] {\small$\Omega_{30}$};
		%\coordinate (e) at (-0.2,-2);
%		\fill (e) circle (0pt) node[left] {\small$\Omega_{5}$};
		\coordinate (f) at (2.2,-0.7);
		\fill (f) circle (0pt) node[right] {\small$\Omega_{40}$};
		%\coordinate (A) at (2,2.6);
%		\coordinate (B) at (2,-2.6);
%		\coordinate (C) at (-0.616996232,1.4);
%		\coordinate (D) at (-0.616996232,-1.4);
%		\coordinate (E) at (1.4,1);
%		\coordinate (F) at (1.4,-1);
%		\coordinate (G) at (-1.8,2);
%		\coordinate (H) at (-1.8,-2);
	%	\fill (A) circle (1pt) node[right] {$z_1$};
%	\fill (B) circle (1pt) node[right] {$\bar{z}_1$};
%	\fill (C) circle (1pt) node[left] {$z_2$};
%	\fill (D) circle (1pt) node[left] {$\bar{z}_2$};
%	\fill (E) circle (1pt) node[right] {$z_3$};
%	\fill (F) circle (1pt) node[right] {$\bar{z}_3$};
%	\fill (G) circle (1pt) node[left] {$z_4$};
%	\fill (H) circle (1pt) node[left] {$\bar{z}_4$};
		\end{tikzpicture}
\caption{The new regions $\Omega(\xi)(\xi\in(-\infty,-\frac{1}{4})\cup(2,\infty))$ and the contours $\Sigma_k$. }\label{Fig-2}
\end{figure}
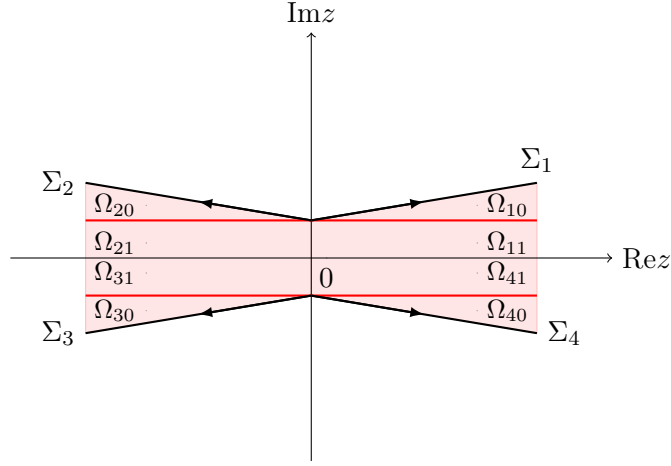

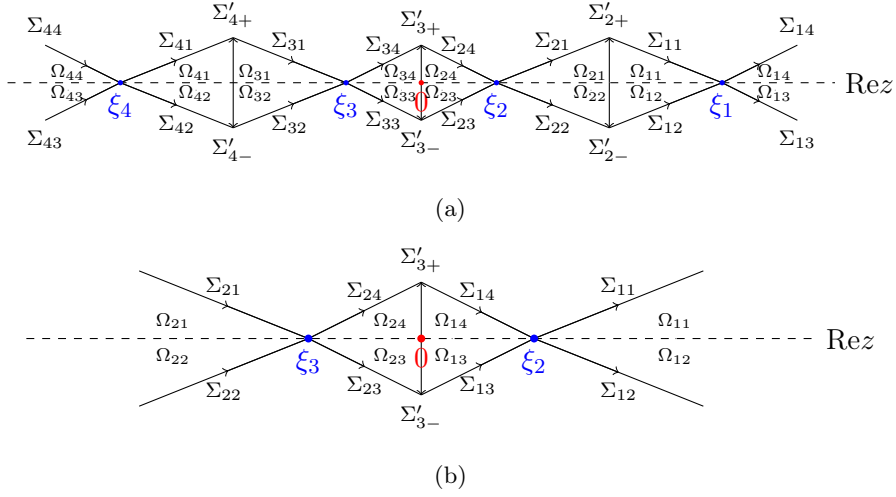
\begin{figure}
\subfigure[]{
\begin{tikzpicture}
\draw(-4,0)--(-5,0.5)node[above]{\scriptsize$\Sigma_{44}$};
\draw[-<](-4,0)--(-4.5,0.25);
\draw(-4,0)--(-2.5,0.6);
\draw[->](-4,0)--(-3.25,-0.3)node[below]{\scriptsize$\Sigma_{42}$};
\draw(-4,0)--(-5,-0.5)node[below]{\scriptsize$\Sigma_{43}$};
\draw[->](-4,0)--(-3.25,0.3)node[above]{\scriptsize$\Sigma_{41}$};
\draw(-4,0)--(-2.5,-0.6);
\draw[-<](-4,0)--(-4.5,-0.25);
\draw(-1,0)--(0,0.5);
\draw[->](-1,0)--(-0.5,0.25)node[above]{\scriptsize$\Sigma_{34}$};
\draw(-1,0)--(-2.5,0.6);
\draw[-<](-1,0)--(-1.75,-0.3)node[below]{\scriptsize$\Sigma_{32}$};
\draw(-1,0)--(0,-0.5);
\draw[-<](-1,0)--(-1.75,0.3)node[above]{\scriptsize$\Sigma_{31}$};
\draw(-1,0)--(-2.5,-0.6);
\draw[->](-1,0)--(-0.5,-0.25)node[below]{\scriptsize$\Sigma_{33}$};
\draw[dashed](-5.5,0)--(5.5,0)node[right]{ Re$z$};
\draw(1,0)--(0,0.5);
\draw[-<](1,0)--(0.5,0.25)node[above]{\scriptsize$\Sigma_{24}$};
\draw(1,0)--(2.5,0.6);
\draw[->](1,0)--(1.75,-0.3)node[below]{\scriptsize$\Sigma_{22}$};
\draw(1,0)--(0,-0.5);
\draw[->](1,0)--(1.75,0.3)node[above]{\scriptsize$\Sigma_{21}$};
\draw(1,0)--(2.5,-0.6);
\draw[-<](1,0)--(0.5,-0.25)node[below]{\scriptsize$\Sigma_{23}$};
\draw(4,0)--(5,0.5)node[above]{\scriptsize$\Sigma_{14}$};
\draw[->](4,0)--(4.5,0.25);
\draw(4,0)--(2.5,0.6);
\draw[-<](4,0)--(3.25,-0.3)node[below]{\scriptsize$\Sigma_{12}$};
\draw(4,0)--(5,-0.5)node[below]{\scriptsize$\Sigma_{13}$};
\draw[-<](4,0)--(3.25,0.3)node[above]{\scriptsize$\Sigma_{11}$};
\draw(4,0)--(2.5,-0.6);
\draw[->](4,0)--(4.5,-0.25);
\draw[->](2.5,0)--(2.5,0.6)node[above]{\scriptsize$\Sigma_{2+}'$};
\draw[->](2.5,0)--(2.5,-0.6)node[below]{\scriptsize$\Sigma_{2-}'$};
\draw[->](-2.5,0)--(-2.5,0.6)node[above]{\scriptsize$\Sigma_{4+}'$};
\draw[->](-2.5,0)--(-2.5,-0.6)node[below]{\scriptsize$\Sigma_{4-}'$};
\draw[->](0,0)--(0,0.5)node[above]{\scriptsize$\Sigma_{3+}'$};
\draw[->](0,0)--(0,-0.5)node[below]{\scriptsize$\Sigma_{3-}'$};
\coordinate (I) at (0,0);
\fill[red] (I) circle (1pt) node[below] {$0$};
\coordinate (A) at (-4,0);
\fill[blue] (A) circle (1pt) node[below] {$\xi_4$};
\coordinate (b) at (-1,0);
\fill[blue] (b) circle (1pt) node[below] {$\xi_3$};
\coordinate (e) at (4,0);
\fill[blue] (e) circle (1pt) node[below] {$\xi_1$};
\coordinate (f) at (1,0);
\fill[blue] (f) circle (1pt) node[below] {$\xi_2$};
\coordinate (ke) at (4.7,0.1);
\fill (ke) circle (0pt) node[below] {\tiny$\Omega_{13}$};
\coordinate (k1e) at (4.7,-0.1);
\fill (k1e) circle (0pt) node[above] {\tiny$\Omega_{14}$};
\coordinate (le) at (3,0.1);
\fill (le) circle (0pt) node[below] {\tiny$\Omega_{12}$};
\coordinate (l1e) at (3,-0.1);
\fill (l1e) circle (0pt) node[above] {\tiny$\Omega_{11}$};
\coordinate (n2) at (0.27,0.1);
\fill (n2) circle (0pt) node[below] {\tiny$\Omega_{23}$};
\coordinate (n12) at (0.27,-0.1);
\fill (n12) circle (0pt) node[above] {\tiny$\Omega_{24}$};
\coordinate (m2) at (2.25,0.1);
\fill (m2) circle (0pt) node[below] {\tiny$\Omega_{22}$};
\coordinate (m12) at (2.25,-0.1);
\fill (m12) circle (0pt) node[above] {\tiny$\Omega_{21}$};
\coordinate (k) at (-4.7,0.1);
\fill (k) circle (0pt) node[below] {\tiny$\Omega_{43}$};
\coordinate (k1) at (-4.7,-0.1);
\fill (k1) circle (0pt) node[above] {\tiny$\Omega_{44}$};
\coordinate (l) at (-3,0.1);
\fill (l) circle (0pt) node[below] {\tiny$\Omega_{42}$};
\coordinate (l1) at (-3,-0.1);
\fill (l1) circle (0pt) node[above] {\tiny$\Omega_{41}$};
\coordinate (n) at (-0.27,0.1);
\fill (n) circle (0pt) node[below] {\tiny$\Omega_{33}$};
\coordinate (n1) at (-0.27,-0.1);
\fill (n1) circle (0pt) node[above] {\tiny$\Omega_{34}$};
\coordinate (m) at (-2.2,0.1);
\fill (m) circle (0pt) node[below] {\tiny$\Omega_{32}$};
\coordinate (m1) at (-2.2,-0.1);
\fill (m1) circle (0pt) node[above] {\tiny$\Omega_{31}$};
%\coordinate (c) at (-2,0);
%\fill[red] (c) circle (1pt) node[below] {\scriptsize$-1$};
%\coordinate (d) at (2,0);
%\fill[red] (d) circle (1pt) node[below] {\scriptsize$1$};
\end{tikzpicture}
\label{case1}}
\subfigure[]{
\begin{tikzpicture}[scale=1.5]
%\draw(-4,0)--(-5,0.5)node[above]{\scriptsize$\Sigma_{44}$};
%\draw[-<](-4,0)--(-4.5,0.25);
%\draw(-4,0)--(-2.5,0.6);
%\draw[->](-4,0)--(-3.25,-0.3)node[below]{\scriptsize$\Sigma_{42}$};
%\draw(-4,0)--(-5,-0.5)node[below]{\scriptsize$\Sigma_{43}$};
%\draw[->](-4,0)--(-3.25,0.3)node[above]{\scriptsize$\Sigma_{41}$};
%\draw(-4,0)--(-2.5,-0.6);
%\draw[-<](-4,0)--(-4.5,-0.25);
\draw(-1,0)--(0,0.5);
\draw[->](-1,0)--(-0.5,0.25)node[above]{\scriptsize$\Sigma_{24}$};
\draw(-1,0)--(-2.5,0.6);
\draw[-<](-1,0)--(-1.75,-0.3)node[below]{\scriptsize$\Sigma_{22}$};
\draw(-1,0)--(0,-0.5);
\draw[-<](-1,0)--(-1.75,0.3)node[above]{\scriptsize$\Sigma_{21}$};
\draw(-1,0)--(-2.5,-0.6);
\draw[->](-1,0)--(-0.5,-0.25)node[below]{\scriptsize$\Sigma_{23}$};
\draw[dashed](-3.5,0)--(3.5,0)node[right]{ Re$z$};
\draw(1,0)--(0,0.5);
\draw[-<](1,0)--(0.5,0.25)node[above]{\scriptsize$\Sigma_{14}$};
\draw(1,0)--(2.5,0.6);
\draw[->](1,0)--(1.75,-0.3)node[below]{\scriptsize$\Sigma_{12}$};
\draw(1,0)--(0,-0.5);
\draw[->](1,0)--(1.75,0.3)node[above]{\scriptsize$\Sigma_{11}$};
\draw(1,0)--(2.5,-0.6);
\draw[-<](1,0)--(0.5,-0.25)node[below]{\scriptsize$\Sigma_{13}$};
%\draw(4,0)--(5,0.5)node[above]{\scriptsize$\Sigma_{14}$};
%\draw[->](4,0)--(4.5,0.25);
%\draw(4,0)--(2.5,0.6);
%\draw[-<](4,0)--(3.25,-0.3)node[below]{\scriptsize$\Sigma_{12}$};
%\draw(4,0)--(5,-0.5)node[below]{\scriptsize$\Sigma_{13}$};
%\draw[-<](4,0)--(3.25,0.3)node[above]{\scriptsize$\Sigma_{11}$};
%\draw(4,0)--(2.5,-0.6);
%\draw[->](4,0)--(4.5,-0.25);
%\draw[->](2.5,0)--(2.5,0.6)node[above]{\scriptsize$\Sigma_{2+}'$};
%\draw[->](2.5,0)--(2.5,-0.6)node[below]{\scriptsize$\Sigma_{2-}'$};
%\draw[->](-2.5,0)--(-2.5,0.6)node[above]{\scriptsize$\Sigma_{4+}'$};
%\draw[->](-2.5,0)--(-2.5,-0.6)node[below]{\scriptsize$\Sigma_{4-}'$};
\draw[->](0,0)--(0,0.5)node[above]{\scriptsize$\Sigma_{3+}'$};
\draw[->](0,0)--(0,-0.5)node[below]{\scriptsize$\Sigma_{3-}'$};
\coordinate (I) at (0,0);
\fill[red] (I) circle (1pt) node[below] {$0$};
%\coordinate (A) at (-4,0);
%\fill[blue] (A) circle (1pt) node[below] {$\xi_4$};
\coordinate (b) at (-1,0);
\fill[blue] (b) circle (1pt) node[below] {$\xi_3$};
%\coordinate (e) at (4,0);
%\fill[blue] (e) circle (1pt) node[below] {$\xi_1$};
\coordinate (f) at (1,0);
\fill[blue] (f) circle (1pt) node[below] {$\xi_2$};
%\coordinate (ke) at (4.7,0.1);
%\fill (ke) circle (0pt) node[below] {\tiny$\Omega_{13}$};
%\coordinate (k1e) at (4.7,-0.1);
%\fill (k1e) circle (0pt) node[above] {\tiny$\Omega_{14}$};
%\coordinate (le) at (3,0.1);
%\fill (le) circle (0pt) node[below] {\tiny$\Omega_{12}$};
%\coordinate (l1e) at (3,-0.1);
%\fill (l1e) circle (0pt) node[above] {\tiny$\Omega_{11}$};
\coordinate (n2) at (0.27,0);
\fill (n2) circle (0pt) node[below] {\tiny$\Omega_{13}$};
\coordinate (n12) at (0.27,0);
\fill (n12) circle (0pt) node[above] {\tiny$\Omega_{14}$};
\coordinate (m2) at (2.25,0);
\fill (m2) circle (0pt) node[below] {\tiny$\Omega_{12}$};
\coordinate (m12) at (2.25,0);
\fill (m12) circle (0pt) node[above] {\tiny$\Omega_{11}$};
%\coordinate (k) at (-4.7,0.1);
%\fill (k) circle (0pt) node[below] {\tiny$\Omega_{43}$};
%\coordinate (k1) at (-4.7,-0.1);
%\fill (k1) circle (0pt) node[above] {\tiny$\Omega_{44}$};
%\coordinate (l) at (-3,0.1);
%\fill (l) circle (0pt) node[below] {\tiny$\Omega_{42}$};
%\coordinate (l1) at (-3,-0.1);
%\fill (l1) circle (0pt) node[above] {\tiny$\Omega_{41}$};
\coordinate (n) at (-0.27,0);
\fill (n) circle (0pt) node[below] {\tiny$\Omega_{23}$};
\coordinate (n1) at (-0.27,0);
\fill (n1) circle (0pt) node[above] {\tiny$\Omega_{24}$};
\coordinate (m) at (-2.2,0);
\fill (m) circle (0pt) node[below] {\tiny$\Omega_{22}$};
\coordinate (m1) at (-2.2,0);
\fill (m1) circle (0pt) node[above] {\tiny$\Omega_{21}$};
%\coordinate (c) at (-2,0);
%\fill[red] (c) circle (1pt) node[below] {\scriptsize$-1$};
%\coordinate (d) at (2,0);
%\fill[red] (d) circle (1pt) node[below] {\scriptsize$1$};
\end{tikzpicture}
\label{case2}}
\caption{Figure (a) and (b) are corresponding to the  $\xi\in(-\frac{1}{4},0)$ and  $\xi\in(0,2)$ respectively. The regions
$\Omega_{jk}$ are of the boundaries $\Sigma_{jk}$.}
\end{figure}

In what follows, we introduce extensions  of the off-diagonal entries of jump matrices of \eqref{RH-2-V1} and \eqref{RH-2-V1+1}. To approach this purpose, we construct matrix function $R^{(2)}(z,\xi)$ and control the norm of $\bar{\partial}R^{(2)}(z,\xi)$ to  ensure that $\bar{\partial}$-contribution has little impact on the long-time asymptotic solution of $q(x,t)$. It is observed that $\theta(z)$ has different properties in different regions, so we have to construct different  $R^{(2)}(z,\xi)$ in different regions. Next, we construct $R^{(2)}(z,\xi)$ as \\
$\bullet$ For $\xi\in(-\infty,-\frac{1}{4})$,
\begin{align}\label{R2-xi-1}
 R^{(2)}(z,\xi)=\left\{\begin{aligned}
&\left(
  \begin{array}{cc}
    1 & R_{j}(\xi)e^{-2it\theta(z)}  \\
    0 & 1 \\
  \end{array}
\right), ~~&z\in\Omega_{j},~~j=1,2,\\
&\left(
  \begin{array}{cc}
    1 & 0 \\
    R_{j}(\xi)e^{-2it\theta} & 1 \\
  \end{array}
\right), ~~&z\in\Omega_{j},~~j=3,4,\\
&\left(
  \begin{array}{cc}
    1 & 0 \\
    0 & 1 \\
  \end{array}
\right),~~ &z\in~~elsewhere.
\end{aligned}
\right.
\end{align}
$\bullet$ For $\xi\in(2,\infty)$,
\begin{align}\label{R2-xi-4}
 R^{(2)}(z,\xi)=\left\{\begin{aligned}
&\left(
  \begin{array}{cc}
    1 &  0 \\
    R_{j}(\xi)e^{2it\theta(z)} & 1 \\
  \end{array}
\right), ~~&z\in\Omega_{j},~~j=1,2,\\
&\left(
  \begin{array}{cc}
    1 & R_{j}(\xi)e^{-2it\theta} \\
    0 & 1 \\
  \end{array}
\right), ~~&z\in\Omega_{j},~~j=3,4,\\
&\left(
  \begin{array}{cc}
    1 & 0 \\
    0 & 1 \\
  \end{array}
\right),~~ &z\in~~elsewhere.
\end{aligned}
\right.
\end{align}
In the above formulae,  $R_{j}(\xi)(j=1,2,3,4)$ is defined  in the following proposition.
\begin{prop}\label{R-property-1}
There exist functions $R_{j}: \bar{\Omega}_{j} \rightarrow C, j= 1, 2, 3, 4$ with boundary values such that,\\
$\bullet$ for $\xi\in(-\infty,-\frac{1}{4})$,
\begin{align*}
&R_{j}(z,\xi)(j=1,2)=\left\{\begin{aligned}&-\frac{\overline{r(z)}}{1-|r(z)|^{2}}T_{+}^{2}(z) \triangleq p_j(z,\xi)T_{+}^{2}(z), &z\in\mathbb{R},\\
&-\frac{\overline{r(0)}}{1-|r(0)|^{2}}T_{0}^{2}(0)z^{2i\nu(0)}, &z\in\Sigma_{1}\cup\Sigma_{2},
\end{aligned}\right.\\
&R_{j}(z,\xi)(j=3,4)=\left\{\begin{aligned}&\frac{r(z)}{1-|r(z)|^{2}}T_{-}^{-2}(z)\triangleq p_j(z,\xi)T_{-}^{-2}(z), &z\in\mathbb{R},\\
&\frac{r(0)}{1-|r(0)|^{2}}T_{0}^{-2}(0)z^{-2i\nu(0)}, &z\in\Sigma_{3}\cup\Sigma_{4},
\end{aligned}\right.
\end{align*}
$\bullet$ for $\xi\in(2,\infty)$,
\begin{align*}
&R_{j}(z,\xi)(j=1,2)=\left\{\begin{aligned}&-r(z)T^{-2}(z) \triangleq p_j(z,\xi)T^{-2}(z), &z\in\mathbb{R},\\
&-r(0)T_0^{-2}(0)z^{-2i\nu(0)}, &z\in\Sigma_{1}\cup\Sigma_{2},
\end{aligned}\right.\\
&R_{j}(z,\xi)(j=3,4)=\left\{\begin{aligned}&-\overline{r(z)}T^{2}(z)\triangleq p_j(z,\xi)T^{2}(z), &z\in\mathbb{R},\\
&-\overline{r(0)}T_{0}^{2}(0)z^{2i\nu(0)}, &z\in\Sigma_{3}\cup\Sigma_{4}.
\end{aligned}\right.
\end{align*}

And for constant $c_1=c_1(q_{0})$,  $R_{j}(j=1,2)$ possesses the following properties:
\begin{align}\label{R-estimate}
\begin{split}
&|R_{j}(z)|\leq \sin^2(\frac{\pi}{2\phi}\arg(z-i\varepsilon_0))+\langle Rez\rangle^{-\frac{1}{2}}, \quad z\in\Omega_{j0},~~j=1,2,3,4.\\
&|\bar{\partial}R_{j}(z)|\leq c_1|z-i\varepsilon_0|^{-1/2}+c_1|p_j'(Rez)|, \quad z\in\Omega_{j0},~~j=1,2,3,4.\\
&\bar{\partial}R_{j}(z)=0,\quad z\in elsewhere,~~or~ dist(z,\mathcal{Z}\cup\bar{\mathcal{Z}})\leq \rho/3,
\end{split}
\end{align}
where $\langle z\rangle=\sqrt{1+z^2}$.

\end{prop}
\begin{proof}
The thinking method to prove the results
%The proof idea of the results
in Proposition \ref{R-property-1} is similar to that in \cite{AIHP,Li-cgNLS}. So we omit it.
\end{proof}

For the cases $\xi\in(-\frac{1}{4},0)$ and $\xi\in(0,2)$, we construct $R^{(2)}(z,\xi)$ as \\
\begin{align}\label{R2-xi-1}
 R^{(2)}(z,\xi)=\left\{\begin{aligned}
&\left(
  \begin{array}{cc}
    1 & R_{kj}(\xi)e^{-2it\theta(z)}  \\
    0 & 1 \\
  \end{array}
\right), ~~&z\in\Omega_{kj},~~j=2,4,~k=1,\ldots,n(\xi),\\
&\left(
  \begin{array}{cc}
    1 & 0 \\
    R_{kj}(\xi)e^{-2it\theta} & 1 \\
  \end{array}
\right), ~~&z\in\Omega_{kj},~~j=1,3,~k=1,\ldots,n(\xi),\\
&\left(
  \begin{array}{cc}
    1 & 0 \\
    0 & 1 \\
  \end{array}
\right),~~ &z\in~~elsewhere.
\end{aligned}
\right.
\end{align}
In the above formulae,  $R_{kj}(\xi)(j=1,2,3,4;k=1,\ldots,n(\xi))$ is defined  in the following proposition.

\begin{prop}\label{R-property-2}
There exist functions $R_{kj}: \bar{\Omega}_{kj} \rightarrow C, j= 1, 2, 3, 4,k=1,\ldots,n(\xi)$ with boundary values such that
\begin{align*}
&R_{k1}(z,\xi)=\left\{\begin{aligned}&p_{k1}(z,\xi)T^{-2}(z), &z\in I_{k1},\\
&p_{k1}(\xi_k,\xi)T_{k}^{-2}(\xi_k)(z-\xi_k)^{-2i\nu(\xi_k)}, &z\in\Sigma_{k1},
\end{aligned}\right.\\
&R_{k2}(z,\xi)=\left\{\begin{aligned}&p_{k2}(z,\xi)T^{2}(z), &z\in I_{k2},\\
&p_{k2}(\xi_k,\xi)T_{k}^{2}(\xi_k)(z-\xi_k)^{2i\nu(\xi_k)}, &z\in\Sigma_{k2},
\end{aligned}\right.\\
&R_{k3}(z,\xi)=\left\{\begin{aligned}&p_{k3}(z,\xi)T_-^{-2}(z), &z\in I_{k3},\\
&p_{k3}(\xi_k,\xi)T_{k}^{-2}(\xi_k)(z-\xi_k)^{-2i\nu(\xi_k)}, &z\in\Sigma_{k3},
\end{aligned}\right.\\
&R_{k4}(z,\xi)=\left\{\begin{aligned}&p_{k4}(z,\xi)T_+^{2}(z), &z\in I_{k4},\\
&p_{k4}(\xi_k,\xi)T_{k}^{2}(\xi_k)(z-\xi_k)^{2i\nu(\xi_k)}, &z\in\Sigma_{k4},
\end{aligned}\right.
\end{align*}
where $p_{kj}(z,\xi)$  are defined as
\begin{align*}
p_{k1}(z,\xi)&=-r(z),~~p_{k2}(z,\xi)=-\overline{r(z)},\\
  p_{k3}(z,\xi)&=\frac{r(z)}{1-|r(z)|^2},~~p_{k4}(z,\xi)= \frac{\overline{r(z)}}{1-|r(z)|^2},
\end{align*}
and $I_{kj}(\xi)$ are defined as\\
$\bullet$ for $\xi\in(-\frac{1}{4},0)$,
\begin{align*}
  I_{11}(\xi)=I_{12}(\xi)=\left(\frac{\xi_1+\xi_2}{2},\xi_1\right),~~ I_{13}(\xi)=I_{14}(\xi)=(\xi_1,\infty),\\
  I_{21}(\xi)=I_{22}(\xi)=\left( \xi_2,\frac{\xi_1+\xi_2}{2}\right),~~ I_{23}(\xi)=I_{24}(\xi)=(0,\xi_2),\\
  I_{31}(\xi)=I_{32}(\xi)=\left(\frac{\xi_3+\xi_4}{2},\xi_3\right),~~ I_{33}(\xi)=I_{34}(\xi)=(\xi_3,0),\\
  I_{41}(\xi)=I_{42}(\xi)=\left(\xi_4,\frac{\xi_3+\xi_4}{2}\right),~~ I_{43}(\xi)=I_{44}(\xi)=(-\infty,\xi_4),
\end{align*}
$\bullet$  for $\xi\in(0,2)$,
\begin{align*}
  I_{11}(\xi)=I_{12}(\xi)=(\xi_2,\infty),~~ I_{13}(\xi)=I_{14}(\xi)=(0,\xi_2),\\
  I_{21}(\xi)=I_{22}(\xi)=(-\infty, \xi_3),~~ I_{23}(\xi)=I_{24}(\xi)=(\xi_3,0).
\end{align*}
And  for constant $c_2=c_2(q_0)$, $R_{kj}$ possess  the following properties:
\begin{align}\label{R-estimate-2}
\begin{split}
&|R_{kj}(z)|\leq \sin^2(\frac{\pi}{2\phi}\arg(z-\xi_k))+\langle Rez\rangle^{-\frac{1}{2}}, \quad z\in\Omega_{kj},~~j=1,2,3,4.\\
&|\bar{\partial}R_{kj}(z)|\leq c_2|z-\xi_k|^{-1/2}+c_2|p_{kj}'(Rez)|, \quad z\in\Omega_{kj}.\\
&\bar{\partial}R_{kj}(z)=0,\quad z\in elsewhere,~~or~ dist(z,\mathcal{Z}\cup\bar{\mathcal{Z}})\leq \rho/3.
\end{split}
\end{align}

\end{prop}
\begin{proof}
The thinking method to prove the results
%The proof idea of the results
in Proposition \ref{R-property-2} is similar to that in \cite{AIHP,Li-cgNLS}. So we omit it.
\end{proof}

Then, by using $R^{(2)}$, we construct a transformation
\begin{align}\label{Trans-2}
M^{(2)}(y,t,z)=M^{(1)}(y,t,z)R^{(2)}(z,\xi),
\end{align}
and  derive the following mixed $\bar{\partial}$-RH problem for $M^{(2)}(y,t,z)$.

\begin{RHP}\label{RH-3}
Find a matrix function $M^{(2)}(y,t,z)$ with the following properties:
\begin{itemize}
  \item $M^{(2)}(y,t,z)$ is meromorphic in $\mathbb{C}\setminus \Sigma^{(\xi,2)}$;
  \item $M^{(2)}(y,t;-z)=M^{(2)}(y,t;z)\sigma_{1}$;
  \item $M^{(2)}(y,t;z)\to(1~~1)$ as $z\rightarrow \infty$;
  \item For $z\in \Sigma^{(\xi,2)} $, the boundary values $M^{(1)}_{\pm}(y,t,z)$ satisfy the jump relationship $M^{(2)}_{+}(y,t,z)=M^{(2)}_{-}(y,t,z)V^{(1)}(z)$, where for $\xi\in(-\infty,-\frac{1}{4})$,
      \begin{align}\label{RH-3-V2}
       V^{(2)}(\xi,z)=\left\{\begin{aligned}
       &\left(
    \begin{array}{cc}
    1 &  -\frac{\overline{r(0)}}{1-|r(0)|^2}T^{2}_0(0)z^{2i\nu(0)}e^{-2it\theta(z)} \\
    0 & 1 \\
     \end{array}
   \right),~~z\in\Sigma_j,~~j=1,2,\\
   &\left(
    \begin{array}{cc}
    1 & 0 \\
    \frac{r(0)}{1-|r(0)|^2}T^{-2}_{0}z^{-2i\nu(0)}e^{2it\theta(z)} & 1 \\
     \end{array}
   \right),~~z\in\Sigma_j,~~j=3,4,\\
   &\left(
        \begin{array}{cc}
      1 & -\frac{z-iImz_{j}}{i\gamma_{j}e^{2it\theta(iImz_{j})}}T(z)^{2} \\
      0 & 1 \\
        \end{array}
      \right), ~~|z-iImz_{j}|= \rho,~ j\neq j_0,\\
   &\left(
        \begin{array}{cc}
      1 & 0 \\
      -\frac{z+iImz_{j}}{i\gamma_{j}e^{-2it\theta(-iImz_{j})}}T(z)^{-2} & 1 \\
        \end{array}
      \right),~~|z+iImz_{j}|= \rho,~ j\neq j_0,
   \end{aligned}\right.
  \end{align}
  for $\xi\in(-\frac{1}{4},2)$,
  \begin{equation}\label{RH-3-V2+1}
       V^{(2)}(z)=\left\{\begin{aligned}
       &\left(
    \begin{array}{cc}
    1 &  R_{kj}(z,\xi)\big|_{z\in\Sigma_{kj}}e^{-2it\theta(z)} \\
    0 & 1 \\
     \end{array}
   \right),~~z\in\Sigma_{kj},~j=2,4,~~k=1,\ldots,n(\xi),\\
   &\left(
    \begin{array}{cc}
    1 &  0 \\
    R_{kj}(z,\xi)\big|_{z\in\Sigma_{kj}}e^{2it\theta(z)} & 1 \\
     \end{array}
   \right),~~z\in\Sigma_{kj},~j=1,3,~~k=1,\ldots,n(\xi),\\
    &\left(
        \begin{array}{cc}
      1 & R_{k2}\big|_{z\in\Sigma_{k2}}e^{-2it\theta(z)} \\
      0 & 1 \\
        \end{array}
      \right)^{-1}\left(
        \begin{array}{cc}
      1 & R_{(k-1)2}\big|_{z\in\Sigma_{(k-1)2}}e^{-2it\theta(z)} \\
      0 & 1 \\
        \end{array}
      \right), ~~z\in\Sigma'_{k+},~~k~~is ~~even,\\
      &\left(
        \begin{array}{cc}
      1 & 0 \\
      R_{k1}\big|_{z\in\Sigma_{k1}}e^{2it\theta(z)} & 1 \\
        \end{array}
      \right)^{-1}\left(
        \begin{array}{cc}
      1 & 0 \\
      R_{(k-1)1}\big|_{z\in\Sigma_{(k-1)1}}e^{2it\theta(z)} & 1 \\
        \end{array}
      \right), ~~z\in\Sigma'_{k-},~~k~~is ~~even,\\
      &\left(
        \begin{array}{cc}
      1 & R_{34}\big|_{z\in\Sigma_{34}}e^{-2it\theta(z)} \\
      0 & 1 \\
        \end{array}
      \right)^{-1}\left(
        \begin{array}{cc}
      1 & R_{24}\big|_{z\in\Sigma_{24}}e^{-2it\theta(z)} \\
      0 & 1 \\
        \end{array}
      \right), ~~z\in\Sigma'_{3+}(-\frac{1}{4}<\xi<0),
      \end{aligned}\right.
       \end{equation}
       \begin{equation}
       V^{(2)}(z)=\left\{\begin{aligned}
      &\left(
        \begin{array}{cc}
      1 & 0 \\
      R_{33}\big|_{z\in\Sigma_{33}}e^{2it\theta(z)} & 1 \\
        \end{array}
      \right)^{-1}\left(
        \begin{array}{cc}
      1 & 0 \\
      R_{23}\big|_{z\in\Sigma_{23}}e^{2it\theta(z)} & 1 \\
        \end{array}
      \right), ~~z\in\Sigma'_{3-}(-\frac{1}{4}<\xi<0),\\
      &\left(
        \begin{array}{cc}
      1 & R_{24}\big|_{z\in\Sigma_{24}}e^{-2it\theta(z)} \\
      0 & 1 \\
        \end{array}
      \right)^{-1}\left(
        \begin{array}{cc}
      1 & R_{14}\big|_{z\in\Sigma_{14}}e^{-2it\theta(z)} \\
      0 & 1 \\
        \end{array}
      \right), ~~z\in\Sigma'_{3+}(0<\xi<2),\\
      &\left(
        \begin{array}{cc}
      1 & 0 \\
      R_{23}(z,\xi)\big|_{z\in\Sigma_{23}}e^{2it\theta(z)} & 1 \\
        \end{array}
      \right)^{-1}\left(
        \begin{array}{cc}
      1 & 0 \\
      R_{13}(z,\xi)\big|_{z\in\Sigma_{13}}e^{2it\theta(z)} & 1 \\
        \end{array}
      \right), ~~z\in\Sigma'_{3-}(0<\xi<2),\\
   &\left(
        \begin{array}{cc}
      1 & -\frac{z-iImz_{j}}{i\gamma_{j}e^{2it\theta(iImz_{j})}}T(z)^{2} \\
      0 & 1 \\
        \end{array}
      \right), ~~|z-iImz_{j}|= \rho,~ j\neq j_0,\\
   &\left(
        \begin{array}{cc}
      1 & 0 \\
      -\frac{z+iImz_{j}}{i\gamma_{j}e^{-2it\theta(-iImz_{j})}}T(z)^{-2} & 1 \\
        \end{array}
      \right),~~|z+iImz_{j}|= \rho,~ j\neq j_0,
      \end{aligned}\right.
  \end{equation}
  for $\xi\in(2,\infty)$,
      \begin{align}\label{RH-3-V2+2}
       V^{(2)}(\xi,z)=\left\{\begin{aligned}
       &\left(
    \begin{array}{cc}
    1 &  -\overline{r(0)}T^{2}_0(0)z^{2i\nu(0)}e^{-2it\theta(z)} \\
    0 & 1 \\
     \end{array}
   \right),~~z\in\Sigma_j,~~j=3,4,\\
   &\left(
    \begin{array}{cc}
    1 & 0 \\
    -r(0)T^{-2}_{0}z^{-2i\nu(0)}e^{2it\theta(z)} & 1 \\
     \end{array}
   \right),~~z\in\Sigma_j,~~j=1,2,\\
   &\left(
        \begin{array}{cc}
      1 & -\frac{z-iImz_{j}}{i\gamma_{j}e^{2it\theta(iImz_{j})}}T(z)^{2} \\
      0 & 1 \\
        \end{array}
      \right), ~~|z-iImz_{j}|= \rho,~~\kappa_0<Imz_j<\frac{1}{2},~ j\neq j_0,\\
    &\left(
        \begin{array}{cc}
      1 & 0 \\
      -\frac{i\gamma_{j}e^{2it\theta(iImz_{j})}}{z-iImz_{j}}T(z)^{-2} & 1 \\
        \end{array}
      \right), ~~|z-iImz_{j}|= \rho,~~0<Imz_j<\kappa_0,~ j\neq j_0,\\
   &\left(
        \begin{array}{cc}
      1 & 0 \\
      -\frac{z+iImz_{j}}{i\gamma_{j}e^{-2it\theta(-iImz_{j})}}T(z)^{-2} & 1 \\
        \end{array}
      \right), ~~|z+iImz_{j}|= \rho,~~\kappa_0<Imz_j<\frac{1}{2},~ j\neq j_0,\\
      &\left(
        \begin{array}{cc}
      1 & -\frac{i\gamma_{j}e^{-2it\theta(-iImz_{j})}}{z+iImz_{j}}T(z)^{2} \\
      0 & 1 \\
        \end{array}
      \right), ~~|z+iImz_{j}|= \rho,~~0<Imz_j<\kappa_0,~ j\neq j_0;
   \end{aligned}\right.
  \end{align}
  \item $\bar{\partial}$-Derivative: For $z\in\mathbb C$, we obtain
  \begin{align}\label{dbar-M2}
    \bar{\partial}M^{(2)}(y,t,z)=M^{(2)}(y,t,z)\bar{\partial}R^{(2)}(z,\xi),
  \end{align}
  where for $\xi\in(-\infty,-\frac{1}{4})$,
  \begin{align}\label{dbar-R2-1}
   \bar{\partial}R^{(2)}(z,\xi)
   =\left\{\begin{aligned}
&\left(
  \begin{array}{cc}
    1 & \bar{\partial}R_{j}e^{-2it\theta(z)}  \\
    0 & 1 \\
  \end{array}
\right), ~~&z\in\Omega_{j},~~j=1,2,\\
&\left(
  \begin{array}{cc}
    1 & 0 \\
    \bar{\partial}R_{j}e^{2it\theta(z)} & 1 \\
  \end{array}
\right), ~~&z\in\Omega_{j},~~j=3,4,\\
&\left(
  \begin{array}{cc}
    0 & 0 \\
    0 & 0 \\
  \end{array}
\right),~~ &z\in ~~elsewhere,
\end{aligned}
\right.
  \end{align}
  for $\xi\in(2,\infty)$,
  \begin{align}\label{dbar-R2-2}
   \bar{\partial}R^{(2)}(z,\xi)
   =\left\{\begin{aligned}
&\left(
  \begin{array}{cc}
    1 & \bar{\partial}R_{j}e^{-2it\theta(z)}  \\
    0 & 1 \\
  \end{array}
\right), ~~&z\in\Omega_{j},~~j=3,4,\\
&\left(
  \begin{array}{cc}
    1 & 0 \\
    \bar{\partial}R_{j}e^{2it\theta(z)} & 1 \\
  \end{array}
\right), ~~&z\in\Omega_{j},~~j=1,2,\\
&\left(
  \begin{array}{cc}
    0 & 0 \\
    0 & 0 \\
  \end{array}
\right),~~ &z\in ~~elsewhere,
\end{aligned}
\right.
  \end{align}
  for $\xi\in(-\frac{1}{4},2)$,
  \begin{align}\label{dbar-R2-3}
   \bar{\partial}R^{(2)}(z,\xi)
   =\left\{\begin{aligned}
&\left(
  \begin{array}{cc}
    1 & \bar{\partial}R_{kj}e^{-2it\theta(z)}  \\
    0 & 1 \\
  \end{array}
\right), ~~&z\in\Omega_{kj},~~k=1,\ldots,n(\xi),~~j=2,4,\\
&\left(
  \begin{array}{cc}
    1 & 0 \\
    \bar{\partial}R_{kj}e^{2it\theta(z)} & 1 \\
  \end{array}
\right), ~~&z\in\Omega_{kj},~~k=1,\ldots,n(\xi),~~j=1,3,\\
&\left(
  \begin{array}{cc}
    0 & 0 \\
    0 & 0 \\
  \end{array}
\right),~~ &z\in ~~elsewhere.
\end{aligned}
\right.
  \end{align}
  \item  $M^{(2)}(y,t,z)$ has simple poles at each $z_{j}\in \mathcal{Z}(j=j_0)$ at which for $\xi\in(-\infty,2)$,
\begin{align}\label{ResM2}
\begin{split}
\mathop{Res}\limits_{z=z_{j_0}}M^{(2)}(y,t,z)=\lim_{z\rightarrow z_{j_0}}M^{(2)}(y,t,z)\left(\begin{array}{cc}
    0 & i\gamma_{j_0}e^{-2it\theta(z_{j_0})}\frac{1}{T}'(z_{j_0})^{-2} \\
    0 & 0 \\
  \end{array}
\right),\\
\mathop{Res}\limits_{z=\bar{z}_{j_0}}M^{(2)}(y,t,z)=\lim_{z\rightarrow \bar{z}_{j_0}}M^{(2)}(y,t,z)\left(\begin{array}{cc}
    0 & 0 \\
    i\gamma^{-1}_{j_0}e^{2it\theta(\bar{z}_{j_0})}T'(\bar{z}_{j_0})^{-2} & 0 \\
  \end{array}
\right);
\end{split}
\end{align}
For   $\xi\in(2,\infty)$, if $\kappa_0<Imz_{j_0}<\frac{1}{2}$, the residue conditions are the same with \eqref{ResM2}; if $0<Imz_{j_0}<\kappa_0$, the residue conditions are that
\begin{align}\label{ResM2+1}
\begin{split}
\mathop{Res}\limits_{z=z_{j_0}}M^{(2)}(y,t,z)=\lim_{z\rightarrow z_{j_0}}M^{(2)}(y,t,z)\left(\begin{array}{cc}
    0 & 0 \\
    i\gamma_{j_0}e^{2it\theta(z_{j_0})}T^{-2}(z_{j_0}) & 0 \\
  \end{array}
\right),\\
\mathop{Res}\limits_{z=\bar{z}_{j_0}}M^{(2)}(y,t,z)=\lim_{z\rightarrow \bar{z}_{j_0}}M^{(2)}(y,t,z)\left(\begin{array}{cc}
    0 & i\gamma_{j_0}e^{-2it\theta(\bar{z}_{j_0})}T^{2}(\bar{z}_{j_0}) \\
    0 & 0 \\
  \end{array}
\right).
\end{split}
\end{align}
\end{itemize}
\end{RHP}

Then, according to the RH problem \ref{RH-3}, the solution $q(x,t)$ of the initial value problem \eqref{CH-equation} can be expressed in a parametric form
\begin{align}\label{2-solution-1}
\begin{split}
  x(y,t)&=y+\ln\frac{M^{(2)}_1(y,t;\frac{i}{2})} {M^{(2)}_2(y,t;\frac{i}{2})},\\
  q(y,t)&=\frac{1}{2i}\lim_{z\to\frac{1}{2}}\left(\frac{M^{(2)}_1(y,t;z) M^{(2)}_2(y,t;z)}{M^{(2)}_1(y,t;\frac{i}{2})M^{(2)}_2(y,t;\frac{i}{2})}-1\right) \frac{1}{z-\frac{i}{2}}.
\end{split}
\end{align}

\section{Decomposition of the mixed $\bar{\partial}$-RH problem}\label{Decomposition-mixed-RH-problem}

The goal of this section is to decompose the mixed $\bar{\partial}$-RH problem into two parts , including a model RH problem with $\bar{\partial}R^{(2)}(z,\xi)=0$ and a pure $\bar{\partial}$-RH problem with $\bar{\partial}R^{(2)}(z,\xi)\neq0$. We denote $M^{(2)}_{RHP}=M^{(2)}_{RHP}(y,t;z)$ as the solution of the model RH problem, and construct a RH problem for $M^{(2)}_{RHP}$ first.

\begin{RHP}\label{RH-rhp}
Find a matrix value function $M^{(2)}_{RHP}(y,t;z)$, admitting
\begin{itemize}
 \item $M^{(2)}_{RHP}$ is analytic in $\mathbb{C}\backslash(\Sigma^{(\xi,2)}$;
 \item  $M^{(2)}_{RHP}(y,t;-z)=M^{(2)}_{RHP}(y,t;z)\sigma_{1}$;
 \item $M^{(2)}_{RHP,+}(y,t,z)=M^{(2)}_{RHP,-}(y,t,z)V^{(2)}(y,t,z),$ \quad $z\in\Sigma^{(\xi,2)}$, where $V^{(2)}(y,t,z)$ is the same with the jump matrix appeared in RHP \ref{RH-3};
 \item $M^{(2)}_{RHP}(y,t;z)\to(1~~1)$ as $z\rightarrow \infty$;
 \item $\bar{\partial}R^{(2)}(z,\xi)=0$, for $z\in\mathbb C$;
 \item $M^{(2)}_{RHP}(y,t;z)$ possesses the same residue condition with $M^{(2)}(y,t;z)$.
 \end{itemize}
\end{RHP}

Next, if  the existence of the solution of $M^{(2)}_{RHP}(y,t;z)$ can be guaranteed, the RHP \ref{RH-3} can be generated to a pure $\bar{\partial}$-RH problem. The existence of the solution of $M^{(2)}_{RHP}(y,t;z)$ will be proved in the following  sections. Now, assuming that the $M^{(2)}_{RHP}(y,t;z)$ exists, and by introducing a transformation
\begin{align}\label{delate-pure-RHP}
M^{(3)}(y,t;z)=M^{(2)}(y,t;z)M^{(2)}_{RHP}(y,t;z)^{-1},
\end{align}
we derive the following pure $\bar{\partial}$-RH problem.
\begin{RHP}\label{RH-4}
Find a matrix value function $M^{(3)}(y,t;z)$, admitting
\begin{itemize}
 \item $M^{(3)}(y,t;z)$ is continuous with sectionally continuous first partial derivatives in $\mathbb{C}\backslash(\Sigma^{(\xi,2)}$;
 \item   $M^{(3)}(y,t;-z)=M^{(3)}(y,t;z)\sigma_{1}$;
 \item For $z\in \mathbb{C}$, we obtain $\bar{\partial}M^{(3)}(z)=M^{(3)}(z)W^{(3)}(z)$,
       where
       \begin{align}\label{RH-4-dbar}
       W^{(3)}=M_{RHP}^{(2)}(y,t;z)\bar{\partial}R^{(2)}(z,\xi)M_{RHP}^{(2)}(y,t;z)^{-1};
       \end{align}
 \item $M^{(3)}(y,t;z)\to(1~~1)$ as $z\rightarrow \infty$.
 \end{itemize}
\end{RHP}
\begin{proof}
On the basis of  the properties of the $M^{(2)}_{RHP}(y,t;z)$ and $M^{(2)}(y,t;z)$ for RHP \ref{RH-rhp} and RHP \ref{RH-3}, the analytic and asymptotic properties of $M^{(3)}(y,t;z)$ can be derived easily. Noting the fact that $M^{(2)}_{RHP}(y,t;z)$ possesses the same jump matrix with $M^{(2)}(y,t;z)$, we obtain that
\begin{align*}
M^{(3)}_{-}(y,t;z)^{-1}M^{(3)}_{+}(y,t;z)&=M^{(2)}_{RHP,-}(y,t;z) M^{(2)}_{-}(y,t;z)^{-1}M^{(2)}_{+}(y,t;z)M^{(2)}_{RHP,+}(y,t;z)^{-1}\\
&=M^{(2)}_{RHP,-}(y,t;z)V^{2}(z)(M^{(2)}_{RHP,-}(y,t;z)V^{2}(z))^{-1}=\mathbb{I},
\end{align*}
which implies that $M^{(3)}$ has no jump. Also, it is easy to prove that there exists no pole in $M^{(3)}$ by a simple analysis. For details, see \cite{AIHP,Li-cgNLS}.
\end{proof}

Additionally, for $\xi\in(-\infty,-\frac{1}{4})\cup(2,\infty)$, the jump matrix $V^{(2)}(z)$ of RH problem \ref{RH-3} possesses the following estimates.

\begin{prop}\label{v2-estimate-xi-1}
For $\xi\in(-\infty,-\frac{1}{4})\cup(2,\infty)$, the jump matrix $V^{(2)}(z)$ admits that
\begin{align}
||V^{(2)}-\mathbb{I}||_{L^{\infty}(\Sigma_{k})}
=O\left(e^{-2\varepsilon_0c_3t}\right),
\end{align}
where $c_3$ is a positive constant.
\end{prop}
\begin{proof}
By using the properties of $R^{(2)}(z,\xi)$ and some techniques, the above result can be obtain by a directly calculation.
\end{proof}

\begin{cor}
For $\xi\in(-\infty,-\frac{1}{4})\cup(2,\infty)$ and  $1\leqslant p<\infty$, the jump matrix $V^{(2)}(z)$ admits that
\begin{align}
||V^{(2)}-\mathbb{I}||_{L^{p}(\Sigma_{k})}
=K_pe^{-2\varepsilon_0c_3t},
\end{align}
where $K_p$ is a  constant that depends on $p$.
\end{cor}

This proposition implies  that for the cases $\xi\in(-\infty,-\frac{1}{4})\cup(2,\infty)$ the jump matrix $V^{(2)}(z)$ uniformly goes to $\mathbb I$ on $\Sigma_{k}$, so there is only exponentially small error (in $t$) by completely ignoring
the jump condition of $M^{(2)}_{RHP}(y,t;z)$. This proposition inspire us to construct the
solution $M^{(2)}_{RHP}(y,t;z)$ of the RH problem \ref{RH-rhp} in following form
\begin{align}\label{decompose-Mrhp}
M^{(2)}_{RHP}(y,t,z)=\left\{\begin{aligned}
&E(z)M^{R}(y,t,z), &&z\in\mathbb{C}\setminus\mathcal{U}_{\xi},\\
&E(z)M^{R}(y,t,z)M^{loc}(y,t,z), &&z\in \mathcal{U}_{\xi},
\end{aligned} \right.
\end{align}
where $\mathcal{U}_{\xi}= \left\{\begin{aligned}\bigcup_{k=1,2,3,4}\mathcal{U}_{\xi_k},\xi\in(-1/4,0),\\
\bigcup_{k=2,3}\mathcal{U}_{\xi_k},\xi\in(0,2),
\end{aligned} \right.$ with
\begin{align*}
 \mathcal{U}_{\xi_k}=\left\{z:|z-\xi_k|<min\{\frac{|\xi_1-\xi_2|}{2}, \frac{|\xi_2-\xi_3|}{2},\rho/3\}\right\}.
\end{align*}

\begin{rem}
If $\xi\in(-\infty,-\frac{1}{4})\cup(2,\infty)$, the jump matrix $V^{(2)}(z)$ uniformly goes to $\mathbb I$ on $\Sigma_{k}$ and has   jump  only on the circle around poles $z_{j_0}$ which gives rise to  $\mathcal{U}_{\xi}=\emptyset$.
\end{rem}

From the above decomposition and the definition of $\mathcal{U}_{\xi}$,  we know that $M^{(loc)}(z)$ possesses no poles in $\mathcal{U}_{\xi}$.
Additionally, $M^{R}$ solves a model RH problem,  $M^{(loc)}$ can be solved by matching  a known parabolic cylinder model in $\mathcal{U}_{\xi}$, and $E(z)$ is an error function which is a solution of a small-norm Riemann-Hilbert problem.

Furthermore,  for $\xi\in(-\frac{1}{4},2)$, the jump matrix $V^{(2)}(z)$ of RH problem \ref{RH-3} possesses the following estimates.

\begin{prop}\label{v2-estimate-xi-2}
For $\xi\in(-\frac{1}{4},2)$, the jump matrix $V^{(2)}(z)$ admits that
\begin{align}
&||V^{(2)}-\mathbb{I}||_{L^{\infty}(\Sigma_{kj}\setminus\mathcal{U}_{\xi})}
=O\left(e^{-c_4t}\right),\\
&||V^{(2)}-\mathbb{I}||_{L^{\infty}(\Sigma'_{k\pm})}
=O\left(e^{-c'_4t}\right),
\end{align}
where $c_4$ and $c'_4$ are a positive constants.
\end{prop}
\begin{proof}
By using the properties of $R^{(2)}(z,\xi)$ and some techniques, the above results can be obtain by a directly calculation.
\end{proof}

\begin{cor}
For $\xi\in(-\frac{1}{4},2)$ and  $1\leqslant p<\infty$, the jump matrix $V^{(2)}(z)$ admits that
\begin{align}
&||V^{(2)}-\mathbb{I}||_{L^{p}(\Sigma_{kj}\setminus\mathcal{U}_{\xi})}
=K_pe^{-c_4t},\\
&||V^{(2)}-\mathbb{I}||_{L^{p}(\Sigma'_{k\pm})}
=K'_pe^{-c'_4t},
\end{align}
where $K_p$ and $K'_p$ are constants that depend on $p$.
\end{cor}

The above results imply that the jump matrix $V^{(2)}(z)$ uniformly goes to $\mathbb I$ on both $\Sigma_{kj}\setminus\mathcal{U}_{\xi}$
and $\Sigma'_{k\pm}$, so outside the neighborhood  $\mathcal{U}_{\xi}$, there is only exponentially small error (in $t$) by completely
ignoring the jump condition of  $M^{(2)}_{RHP}(y,t;z)$.

\section{Outer model RH problem: $M^{R}$}

In this  section, we construct a model RH problem for $M^{R}$ and prove that the solution of $M^{R}$ can be approximated by a finite sum of soliton solutions.

The model RH problem for $M^{R}$ admits a RH problem as follows:
\begin{RHP}\label{RH-5}
Find a matrix value function $M^{R}(y,t;z)$, admitting
\begin{itemize}
  \item $M^{R}(y,t;z)$ is analytical in $\mathbb{C}\setminus\{z_{j},\bar{z}_{j}\}_{j=j_0}$;
  \item $M^{R}(y,t;z)\to(1~~1)$ as $z\rightarrow \infty$;
  \item $M^{R}(y,t;z)$ has simple poles at  each $\{z_{j},\bar{z}_{j}\}_{j=j_0}$ and admits the same residue condition in RH problem \ref{RH-3} by replacing $M^{(2)}(y,t;z)$ with $M^{R}(y,t;z)$.
\end{itemize}
\end{RHP}

\begin{prop}
For given scattering data $\sigma_d=\left\{z_j,r(z)\right\}_{j=1,2,\ldots,N}$, if  $M^{R}(y,t;z)$ is the solution of RH problem \ref{RH-4}, then  $M^{R}(y,t;z)$ exists and is unique.
\end{prop}

\begin{proof}
  This result can be proved by a simple calculation and the thinking method is similar to \cite{mCH-Fan-AM}, so we omit it.
\end{proof}

\subsection{The asymptotic $N(j_0)$ solition solution}

Although $M^{R}(y,t,z)$ exists and is unique, as $t\to\infty$, not all discrete spectral points contribute to the solution $M^{R}(y,t,z)$. Next, we show that the jump   matrix $V^{(2)}(z)$ on the circles around poles $z_j$ and $\bar{z}_j$ $(j\in\mathcal{N}, j\neq j_0)$ uniformly goes to $\mathbb I$ as $t\to\infty$.

Firstly, we introduce some notations:
\begin{align*}
 \Sigma^{(cir)}=\left\{|z-z_j|=\rho,|z-\bar{z}_j|=\rho\right\}_{j\in\mathcal{N},j\neq j_0}.
\end{align*}
Then, we show that on the contour $\Sigma^{(cir)}$, the jump matrix $V^{(2)}(z)$ satisfies the following estimates.

\begin{prop}\label{V2-Estimate-3+1}
As $t\to\infty$, there exists positive constant $c_4$ such that
\begin{align}
||  V^{(2)}-\mathbb{I}||_{L^{\infty}(\Sigma^{(cir)})}
=O(e^{-2c_4t}),
\end{align}
where $c_4$ is a positive constant.
\end{prop}

\begin{proof}
On the basis of the definition of the jump matrix $V^{(2)}(z,\xi)$ on the circle around the poles $z_j$ and $\bar{z}_j$ $(j\in\mathcal{N}, j\neq j_0)$, this result can be proved by a simple calculation.
\end{proof}

From Proposition \ref{V2-Estimate-3+1}, we know that  the jump condition on $M^{R}(z)$ can be  completely ignored, because there is only exponentially small error (in $t$). Then, we decompose $M^{R}(y,t,z)$ as
\begin{align}\label{Trans-3}
M^{R}(y,t,z)=E^{sol}(y,t,z)M^{R}_{j_0}(y,t,z),
\end{align}
where $E^{sol}(y,t,z)$ is a error function and $M^{R}_{j_0}(y,t,z)$ solves RH problem \ref{RH-rhp} with $V^{(2)}(z)\equiv\mathbb{I}$. Additionally, $E^{sol}(y,t,z)$ is a solution of a small-norm RH problem. Then, RH problem \ref{RH-rhp} is reduced to the following RH problem.

\begin{RHP}\label{RH-6}
Find a matrix value function $M^{R}_{j_0}(y,t,z)$, admitting
\begin{itemize}
 \item $M^{R}_{j_0}(y,t,z)$ is analytic in $\mathbb{C}\setminus \left\{z_j,\bar{z}_j\right\}_{j=j_0}$;
 \item  $M^{R}_{j_0}(y,t,-z)=M^{R}_{j_0}(y,t,z)\sigma_{1}$;
 \item $M^{R}_{j_0}(y,t,z)\to(1~~1)$ as $z\rightarrow \infty$;
 \item $M^{R}_{j_0}(y,t,z)$ possesses the same residue condition with $M^{(2)}(y,t;z)$.
 \end{itemize}
\end{RHP}

\begin{rem}
It is noted that the discrete spectrum $z_j$ are distributed in the interval $i(0,\frac{1}{2})$, for the convenience of calculation and without loss of generality, we assume that there exists only one $j_0$  such that $|\frac{1}{2}-z_{j_0}|<\rho$ for all the cases, i.e., $\xi\in(-\infty,-\frac{1}{4})\cup(-\frac{1}{4},0)\cup(0,2)\cup(2,\infty)$. Furthermore, for the case $\xi\in(2,\infty)$, we assume $|\kappa_0-z_{j}|>\rho$ for any $j\in\mathcal{N}$. For the other cases, i.e., there  exist  more than one $j_0$  such that $|\frac{1}{2}-z_{j_0}|<\rho$ or  $|\kappa_0-z_{j_0}|<\rho$, it is just more computationally complex.
\end{rem}

Then, we can give the following proposition.

\begin{prop}\label{Msol-prop}
The RH problem \ref{RH-6} possesses unique solution. Moreover, $M^{R}_{j_0}(y,t;z)$ has  equivalent solution to the original RH problem \ref{RH-1} with modified scattering data $D_{j_0}=\{r(z)\equiv0,   \{z_{j_0},  \gamma_{j_0}T^{-2}(z_{j_0})\}$ under the condition that $r(z)\equiv0$ as:\\
If there exists   $j_0$  such that $|\frac{1}{2}-z_{j_0}|<\rho$, then  by using symmetry condition of  $ M^{R}_{j_0}(z)$, we have
\begin{align}
 M^{R}_{j_0}(z)=(f(z)~f(-z)),
\end{align}
where
\begin{align*}
  f(z)=\frac{1}{1+\alpha}\left(1+\alpha\frac{z+z_{j_0}}{z-z_{j_0}}\right),
\end{align*}
with $\alpha=\frac{(\gamma_{j_0}T^{-2}(z_{j_0}))^{2}}{2Imz_{j_0}}e^{2t\theta(z_{j_0})}$.
Then, accordance with \eqref{2-solution-1}, the CH equation \eqref{CH-equation} admits a one-soliton solution \cite{CH-Longtime}
 \begin{align}\label{3-solution-1}
\begin{split}
  q(x,t|D_{j_0})&=\frac{32Imz^{2}_{j_0}}{1-4Imz^{2}_{j_0}}\alpha(y,t) \left((1+\alpha(y,t))^2+\frac{16Imz^{2}_{j_0}}{1-4Imz^{2}_{j_0}}\alpha(y,t)\right)^{-1},\\
  x(y,t|D_{j_0})&=y+\ln\frac{1+\alpha(y,t)\frac{1+2Imz_{j_0}}{1-2Imz_{j_0}}} {1+\alpha(y,t)\frac{1-2Imz_{j_0}}{1+2Imz_{j_0}}}\triangleq y +c(x,t|D_{j_0}).
\end{split}
\end{align}
\end{prop}

\begin{proof}
According to the Liouville's theorem, the uniqueness of solution follows  immediately. The soliton solution can be obtained by a directly calculation, (see \cite{CH-Longtime}).
\end{proof}

\subsection{The error function $E^{sol}(y,t,z)$ between $M^{R}(y,t,z)$ and $M^{R}_{j_0}(y,t,z)$ }\label{error-function-Merr}

In this subsection, we are going  to study the error matrix-function $E^{sol}(y,t,z)$. The first step is to prove that the
error function $E^{sol}(y,t,z)$ solves a small norm RH problem. Next, we show that  $E^{sol}(y,t,z)$ can be expanded
asymptotically for large time.
Firstly, on the basis of the decomposition \eqref{Trans-3}, we can derive a RH problem  with respect  to matrix function $E^{sol}(y,t,z)$.

\begin{RHP}\label{RH-7}
Find a matrix-valued function $E^{sol}(y,t,z)$ satisfies that
\begin{itemize}
 \item $E^{sol}(y,t,z)$ is continuous in $\mathbb{C}\setminus\Sigma^{(cir)}$;
 \item $E^{sol}(y,t,z)(y,t;z)\to(1~~1)$ as $z\rightarrow \infty$;.
 \item $E^{sol}_+(y,t,z)=E_-^{sol}(y,t,z)V^{(sol-E)}(z)$, \quad $z\in\Sigma^{(cir)}$, where
\end{itemize}
 \begin{align}\label{VE-Jump-B}
 V^{(sol-E)}(z)=
 M^{R}_{j_0}(y,t,z)V^{(2)}(z,\xi)M^{R}_{j_0}(y,t,z)^{-1}.
 \end{align}
\end{RHP}

The jump matrix  $V^{(sol-E)}(z)$ in RHP \ref{RH-7} admits the following uniformly estimation.
\begin{prop}\label{VE-estimate-prop}
The jump matrix  $V^{(sol-E)}(z)$  admits
\begin{align}\label{VE-estimate}
\big|\big|V^{(sol-E)}(z)-\mathbb{I}\big|\big|_{L^{p}(\Sigma^{(cir)})} =O(e^{-2ct}).
\end{align}
\end{prop}
\begin{proof}
From the Proposition \ref{Msol-prop},  we learn that $M^{R}_{j_0}(y,t,z)$ is bounded on $\Sigma^{(cir)}$. Then, we have
\begin{align}\label{VE-estimate-proof}
||V^{(sol-E)}-\mathbb{I}||_{L^{p}(\Sigma^{(cir)})} =||V^{(2)}-\mathbb{I}||_{L^{p}(\Sigma^{(cir)})}.
\end{align}
Furthermore, by using Proposition \ref{V2-Estimate-3+1}, this result \eqref{VE-estimate} can be directly derived.
\end{proof}

From Proposition \ref{VE-estimate-prop}, we know  that RH problem \ref{RH-7} can be established as a  small-norm RH problem. Therefore, by using a small-norm RH problem, the solution of the RH problem \ref{RH-7} exists and unique \cite{Deift-1994-2,Deift-2003}.
In what follows, we give a briefly description of this process.

Based on  the Beals-Coifman theory, we   decompose  the jump matrix $V^{(sol-E)}$
\begin{align*}
V^{(sol-E)}(z)=(b_{-})^{-1}b_{+}, ~~b_{-}=\mathbb{I}, ~~b_{+}=V^{(sol-E)}(z).
\end{align*}
Define
\begin{align*}
(\omega_{e})_{-}=\mathbb{I}-b_{-},~~(\omega_{e})_{+}=b_{+}+\mathbb{I}, ~~\omega_{e}=(\omega_{e})_{+}+(\omega_{e})_{-}=V^{(sol-E)}(z)-\mathbb{I}.
\end{align*}
Furthermore, we define  the integral operator $C_{\omega_{err}}(L^{\infty}(\Sigma^{(cir)})\rightarrow L^{2}(\Sigma^{(cir)}))$ as
\begin{align*}
C_{\omega_{err}}f(z)=C_{-}(f(\omega_{e})_{+})+C_{+} (f(\omega_{e})_{-})=C_{-}(f(V^{(sol-E)}(z)-\mathbb{I})),
\end{align*}
where $C_{-}$ is the Cauchy projection operator
\begin{align}\label{Cauchy-opera}
C_{-}(f)(z)\lim_{z\rightarrow\Sigma_{-}^{(cir)}}\int_{\Sigma^{(cir)}}\frac{f(s)}{s-z}ds,
\end{align}
and $||C_{-}||_{L^{2}}$ is bounded. Therefore, we derive   the solution of RH problem \ref{RH-7} for $E^{sol}(y,t,z)$  as
\begin{align}\label{E-solution}
E^{sol}(y,t,z)=\mathbb{I}+\frac{1}{2\pi i}\int_{\Sigma^{(cir)}}\frac{\mu_{e}(s) (V^{(sol-E)}(s)-\mathbb{I})}{s-z}\,ds,
\end{align}
where $\mu_{e}\in L^2 (\Sigma^{(cir)})$  is the solution of the following equation
\begin{align}\label{mu-equation}
(1-C_{\omega_{err}})\mu_{e}=\mathbb{I}.
\end{align}
Next,  according to the properties of the Cauchy projection operator $C_{-}$ and  Proposition \ref{VE-estimate-prop}, we obtain
\begin{align}%\label{}
\begin{split}
\|C_{\mu_{e}}\|_{L^2(\Sigma^{(cir)})}& \lesssim\|C_-\|_{L^2(\Sigma^{(cir)})\rightarrow L^2(\Sigma^{(cir)})}\|V^{(err)}-\mathbb{I}\|_{L^{\infty}
(\Sigma^{(cir)})}\\
&\lesssim O(e^{-2ct}).
\end{split}
\end{align}
This  implies that $1-C_{\omega_{err}}$ is invertible.
Furthermore,
\begin{align}\label{mu-tildeE-estimation}
||\mu_{e}||_{L^{2}(\Sigma^{(cir)})}\lesssim \frac{\|C_{\omega_{e}}\|}{1-\|C_{\omega_{e}}\|}\lesssim O(e^{-2ct}).
\end{align}
Hence, $\mu_{e}$ is existence and uniqueness. Now,  the existence of the solution of RH problem \ref{RH-7} for $E^{sol}(y,t,z)$ is guaranteed.

Next, we need to reconstruct the solutions of the CH equation \eqref{CH-equation}.To achieve this goal, the asymptotic behavior of $E^{sol}(y,t,z)$ as $z\rightarrow\infty$ and $z\rightarrow\frac{i}{2}$ need to be evaluated.

\begin{prop}
As $z\rightarrow\infty$, the $E^{sol}(y,t,z)$ defined in \eqref{Trans-3} satisfies
\begin{align}\label{E-estimation}
|E^{sol}(y,t,z)-\mathbb{I}|\lesssim O(e^{-2ct}).
\end{align}
When $z=\frac{i}{2}$,
\begin{align}%\label{E-solution}
E^{sol}(\frac{i}{2})=\mathbb{I}+\frac{1}{2\pi i}\int_{\Sigma^{(cir)}}\frac{(\mathbb I+\mu_{e}(s))(V^{(sol-E)}(s)-\mathbb{I})}{s-\frac{i}{2}}\,ds.
\end{align}
As $z\rightarrow\frac{i}{2}$, $E^{sol}(y,t,z)$ can be expanded as
\begin{align}%\label{E-solution}
E^{sol}(z)=E^{sol}(\frac{i}{2})+E^{sol}_1(z-\frac{i}{2})+ O\left(\left(z-\frac{i}{2}\right)^2\right),
\end{align}
where
\begin{align*}
  E^{sol}_1=\frac{1}{2\pi i}\int_{\Sigma^{(cir)}}\frac{(\mathbb I+\mu_{e}(s))(V^{(sol-E)}(s)-\mathbb{I})}{\left(s-\frac{i}{2}\right)^2}\,ds.
\end{align*}
Additionally, $E^{sol}(\frac{i}{2})$ and $E^{sol}_1$ satisfy the following long-time asymptotic behavior,
\begin{align}\label{E-sol-E-1}
  |E^{sol}(\frac{i}{2})-\mathbb I|\lesssim O(e^{-2ct}),~~ |E^{sol}_1|\lesssim O(e^{-2ct}).
\end{align}
\end{prop}
\begin{proof}
According to \eqref{VE-estimate}, \eqref{E-solution} and \eqref{mu-tildeE-estimation}, we can directly derive the formula \eqref{E-estimation}.
\end{proof}

Then, according to the above results, we have the following corollary.
\begin{cor}
For   $|t|\gg1$, uniformly for $z\in\mathbb{C}$,  $M^{out}(x,t;z)$  is expressed as
\begin{align}\label{Mout-MoutLambda}
M^{R}(z)=M^{R}_{j_0}(z)\left(\mathbb{I}+O(e^{-ct})\right).
\end{align}
\end{cor}

\subsection{Local solvable model near phase points for $\xi\in (-\frac{1}{4},2)$}\label{section-local-model}

Proposition \ref{v2-estimate-xi-2} implies that $V^{(2)}-\mathbb I$ does not have a uniform estimate for large time near the phase point $z=\pm z_0,\pm z_1$(Here, we still use $\pm z_0,\pm z_1$ to represent the stationary points). Therefore,  we need to continue our study near the stationary phase points  in this subsection.

Define  local jump contour as (see Fig. \ref{fig-5})
\begin{align*}
\Sigma^{(loc)}&=\Sigma^{(\xi,2)}\cap\mathcal U_{\xi}.
\end{align*}

\begin{figure}[h]
\subfigure[]{
\begin{tikzpicture}[scale=1.2]
\draw(-4,0)--(-4.8,0.4);
\draw[dashed](0,2)node[right]{ Im$z$}--(0,-2);
\draw[-latex](0,2)--(0,2.1);
\draw[-latex](5.3,0)--(5.4,0);
\draw[-<](-4,0)--(-4.4,0.2);
\draw(-4,0)--(-3.2,0.4);
\draw[->](-4,0)--(-3.6,-0.2);
\draw(-4,0)--(-4.8,-0.4);
\draw[->](-4,0)--(-3.6,0.2);
\draw(-4,0)--(-3.2,-0.4);
\draw[-<](-4,0)--(-4.4,-0.2);
\draw(-1,0)--(-0.2,0.4);
\draw[->](-1,0)--(-0.6,0.2);
\draw(-1,0)--(-1.8,0.4);
\draw[-<](-1,0)--(-1.4,-0.2);
\draw(-1,0)--(-0.2,-0.4);
\draw[-<](-1,0)--(-1.4,0.2);
\draw(-1,0)--(-1.8,-0.4);
\draw[->](-1,0)--(-0.6,-0.2);
\draw[dashed](-5,0)--(5.4,0)node[right]{ Re$z$};
\draw(1,0)--(0.2,0.4);
\draw[-<](1,0)--(0.6,0.2);
\draw(1,0)--(0.2,-0.4);
\draw[->](1,0)--(1.4,-0.2);
\draw(1,0)--(1.8,0.4);
\draw[->](1,0)--(1.4,0.2);
\draw(1,0)--(1.8,-0.4);
\draw[-<](1,0)--(0.6,-0.2);
\draw(4,0)--(4.8,0.4);
\draw[->](4,0)--(4.4,0.2);
\draw(4,0)--(3.2,0.4);
\draw[-<](4,0)--(3.6,-0.2);
\draw(4,0)--(4.8,-0.4);
\draw[-<](4,0)--(3.6,0.2);
\draw(4,0)--(3.2,-0.4);
\draw[->](4,0)--(4.4,-0.2);
\coordinate (I) at (0,0);
\fill[red] (I) circle (1pt) node[below right] {$0$};
\coordinate (A) at (-4,0);
\fill (A) circle (1pt) node[below] {$\xi_4$};
\coordinate (b) at (-1,0);
\fill (b) circle (1pt) node[below] {$\xi_3$};
\coordinate (e) at (4,0);
\fill (e) circle (1pt) node[below] {$\xi_1$};
\coordinate (f) at (1,0);
\fill (f) circle (1pt) node[below] {$\xi_2$};
\coordinate (c) at (-2,0);
\fill[red] (c) circle (1pt) node[below] {\scriptsize$-1$};
\coordinate (d) at (2,0);
\fill[red] (d) circle (1pt) node[below] {\scriptsize$1$};
\end{tikzpicture}
\label{si1}}
\subfigure[]{
\begin{tikzpicture}[scale=1.7]
%\draw(-4,0)--(-4.8,0.4);
\draw[dashed](0,1.5)node[right]{ Im$z$}--(0,-1.5);
\draw[-latex](0,1.5)--(0,1.6);
\draw[-latex](3.3,0)--(3.4,0);
%\draw[-<](-4,0)--(-4.4,0.2);
%\draw(-4,0)--(-3.2,0.4);
%\draw[->](-4,0)--(-3.6,-0.2);
%\draw(-4,0)--(-4.8,-0.4);
%\draw[->](-4,0)--(-3.6,0.2);
%\draw(-4,0)--(-3.2,-0.4);
%\draw[-<](-4,0)--(-4.4,-0.2);
\draw(-1,0)--(-0.2,0.4);
\draw[->](-1,0)--(-0.6,0.2);
\draw(-1,0)--(-1.8,0.4);
\draw[-<](-1,0)--(-1.4,-0.2);
\draw(-1,0)--(-0.2,-0.4);
\draw[-<](-1,0)--(-1.4,0.2);
\draw(-1,0)--(-1.8,-0.4);
\draw[->](-1,0)--(-0.6,-0.2);
\draw[dashed](-3,0)--(3.4,0)node[right]{ Re$z$};
\draw(1,0)--(0.2,0.4);
\draw[-<](1,0)--(0.6,0.2);
\draw(1,0)--(0.2,-0.4);
\draw[->](1,0)--(1.4,-0.2);
\draw(1,0)--(1.8,0.4);
\draw[->](1,0)--(1.4,0.2);
\draw(1,0)--(1.8,-0.4);
\draw[-<](1,0)--(0.6,-0.2);
%\draw(4,0)--(4.8,0.4);
%\draw[->](4,0)--(4.4,0.2);
%\draw(4,0)--(3.2,0.4);
%\draw[-<](4,0)--(3.6,-0.2);
%\draw(4,0)--(4.8,-0.4);
%\draw[-<](4,0)--(3.6,0.2);
%\draw(4,0)--(3.2,-0.4);
%\draw[->](4,0)--(4.4,-0.2);
\coordinate (I) at (0,0);
\fill[red] (I) circle (1pt) node[below right] {$0$};
%\coordinate (A) at (-4,0);
%\fill (A) circle (1pt) node[below] {$\xi_4$};
\coordinate (b) at (-1,0);
\fill (b) circle (1pt) node[below] {$\xi_3$};
%\coordinate (e) at (4,0);
%\fill (e) circle (1pt) node[below] {$\xi_1$};
\coordinate (f) at (1,0);
\fill (f) circle (1pt) node[below] {$\xi_2$};
\coordinate (c) at (-2,0);
\fill[red] (c) circle (1pt) node[below] {\scriptsize$-1$};
\coordinate (d) at (2,0);
\fill[red] (d) circle (1pt) node[below] {\scriptsize$1$};
\end{tikzpicture}
\label{si2}}
\caption{Figures (a) and (b) denote the contour $\Sigma^{\xi,2}\cap\mathcal{U}_{\xi}$ corresponding to the  $-\frac{1}{4}<\xi<0$ and $0<\xi<2$, respectively.}
	\label{fig-5}
\end{figure}
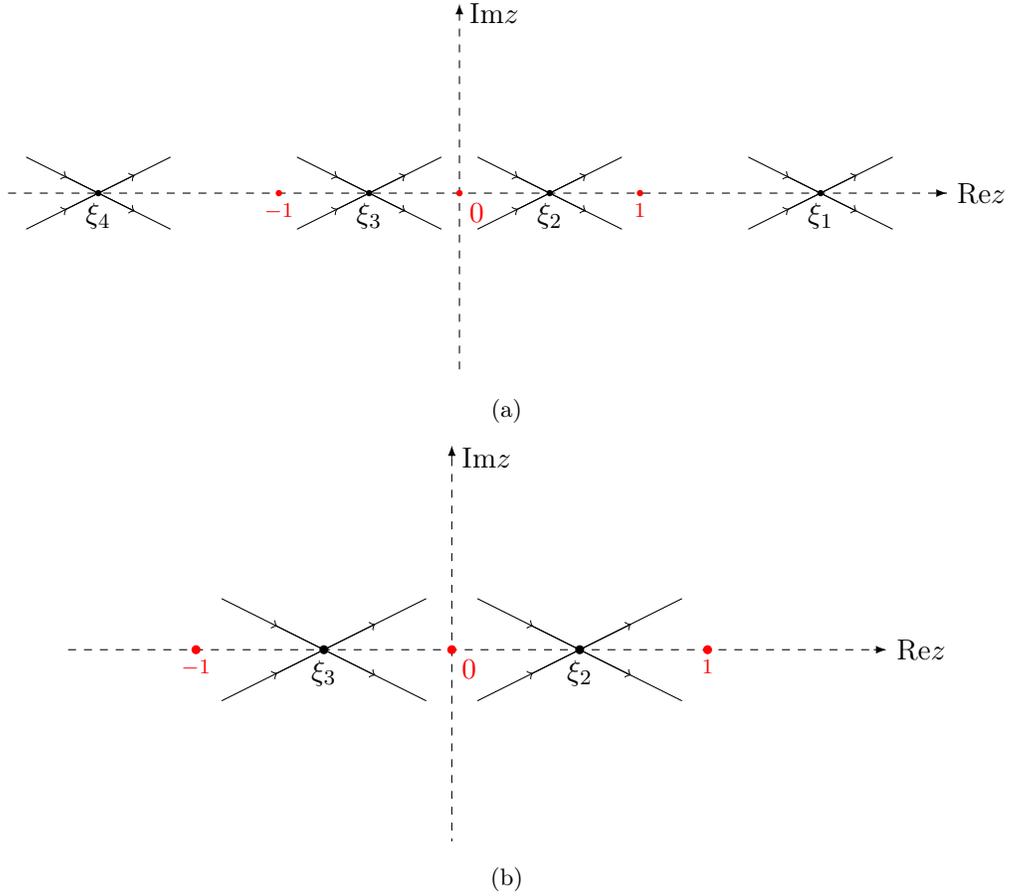

Next, on the basis of the definition of  $\mathcal U_{\xi}$ we find that there are no discrete spectrum in $\mathcal{U}_{\xi}$. Therefore, we have $T(z)=\delta(z)$ and RH problem \ref{RH-rhp} can be reduced to the following model.

\begin{RHP}\label{RH-loc}
Find a matrix value function $M^{(loc)}(y,t,z)$, admitting
\begin{itemize}
 \item $M^{(loc)}(y,t,z)$ is continuous in $\mathbb{C}\setminus(\Sigma^{loc})$.
 \item $M^{(loc)}(y,t;-z)=M^{(loc)}(y,t;z)\sigma_{1}$;
  \item $M^{(loc)}(y,t;z)\to(1~~1)$ as $z\rightarrow \infty$;
 \item $M_+^{(loc)}(y,t,z)=M_{-}^{(loc)}(y,t,z)V^{(loc)}(z),$ ~~ $z\in\Sigma^{(loc)}$, where the jump matrix $V^{(loc)}(z)$ satisfies
  \begin{align}\label{V-loc-jump}
V^{(loc)}(z)=\left\{\begin{aligned}
       &\left(
    \begin{array}{cc}
    1 &  R_{kj}(z,\xi)\big|_{z\in\Sigma_{kj}}e^{-2it\theta(z)} \\
    0 & 1 \\
     \end{array}
   \right),~~z\in\Sigma_{kj}\cap\mathcal U_{\xi},~j=2,4,~~k=1,\ldots,n(\xi),\\
   &\left(
    \begin{array}{cc}
    1 &  0 \\
    R_{kj}(z,\xi)\big|_{z\in\Sigma_{kj}}e^{2it\theta(z)} & 1 \\
     \end{array}
   \right),~~z\in\Sigma_{kj}\cap\mathcal U_{\xi},~j=1,3,~~k=1,\ldots,n(\xi).
   \end{aligned}\right.
\end{align}
\end{itemize}
\end{RHP}

Proposition \ref{v2-estimate-xi-2} implies that the jump matrix $V^{(loc)}(z)$ uniformly goes to $\mathbb I$ outside the neighborhood of $\pm z_0$ and $\pm z_1$, then following the result in \cite{CH-Longtime}, the above RH problem \ref{RH-loc} is solvable. The main
contribution to the $M^{(loc)}(y,t,z)$ comes from a local RH problem near $\pm z_0$ and $\pm z_1$ ( see Fig. \ref{fig-5}). We next describe the process of construction briefly for the solution of the RH problem \ref{RH-loc} (see \cite{CH-Longtime} for the detail).

Firstly, we note that for $\ell=0,1$,
\begin{align}
  \theta(z_{\ell})&=-\frac{16z^2_{\ell}}{(1+4z^2_{\ell})^2},\\
  \theta''(z_{\ell})&=\frac{48z_{\ell}-64z^3_{\ell}}{(1+4z^2_{\ell})^3},
\end{align}
where $\theta(z_{0})>0$ for $\xi\in(-\frac{1}{4},2)$ and $\theta(z_{1})<0$ for $\xi\in(-\frac{1}{4},0)$.

Furthermore, we make a change of coordinates
\begin{align}
  \zeta_{\ell}&=\sqrt{2\theta''(z_{\ell})}(z-z_{\ell}),\\
  z&=z_{\ell}+\frac{\zeta_{\ell}}{\sqrt{2\theta''(z_{\ell})}},
\end{align}
such that
\begin{align*}
  \theta(z)= \theta(z_{\ell})+\frac{1}{4}\zeta^2_{\ell}+O(\zeta_{\ell}^3).
\end{align*}
Set \begin{align}
      r_{\ell}=r(z_{\ell})(\sqrt{2\theta''(z_{\ell})})^{-2i\nu(z_{\ell})} e^{2it\theta(z_{\ell})T_{\ell}^{-2}(z_{\ell})}.
    \end{align}

Then the solution of the RH problem \ref{RH-loc} near the stationary points $z_{\ell}$ is given by
\begin{align}
  M^{(loc)}_{z_{\ell}}(z)=\mathbb I+\frac{1}{\sqrt{2\theta''(z_{\ell})}(z-z_{\ell})}\frac{i}{t^{\frac{1}{2}}} \left(\begin{array}{cc}
                   0 & -\beta_{\ell} \\
                  \bar{ \beta}_{\ell} & 0
                 \end{array}\right)+O(t^{-\alpha}),
\end{align}
where $\alpha\in(\frac{1}{2},1)$, and
\begin{align*}
\beta_{0}&=\sqrt{\nu(z_0)}\exp i\left(\frac{\pi}{4}-\arg(r(z_0))+\arg\Gamma(i\nu(z_0))+2\arg(T_{0}(z_{0}) - \nu(z_0)\log(2t\theta''(z_{0}))-t\theta(z_0) \right),\\
  \beta_{1}&=\sqrt{\nu(z_1)}\exp i\left(-\frac{\pi}{4}-\arg(r(z_1))-\arg\Gamma(i\nu(z_1))+2\arg(T_{1}(z_{1}) + \nu(z_1)\log(2t\theta''(z_{1}))+t\theta(z_1) \right).
\end{align*}

Then, according to the symmetry condition of the vector RH problem,  near the other two stationary points $-z_1$ and $-z_0$, the solution of the RH problem \ref{RH-loc} are that
\begin{align}
  \begin{split}
    M^{(loc)}_{-z_{\ell}}(z)&=\Sigma_1M^{(loc)}_{z_{\ell}}(-z)\Sigma_1\\
    &=\mathbb I-\frac{(-1)^{\ell}}{\sqrt{2\theta''(z_{\ell})}(z+z_{\ell})}\frac{i}{t^{\frac{1}{2}}} \left(\begin{array}{cc}
                   0 & \bar{\beta}_{\ell} \\
                   -\beta_{\ell} & 0
                 \end{array}\right)+O(t^{-\alpha}).
  \end{split}
\end{align}

Finally, for the case $\xi\in(-\frac{1}{4},0)$, the solution of the RH problem \ref{RH-loc} can be derived as
\begin{align}
\begin{split}
   M^{(loc)}(z)=&(1~~1)+\frac{1}{\sqrt{2\theta''(z_{0})}(z-z_{0})} \frac{i}{t^{\frac{1}{2}}} ( -\beta_{0}~~\bar{\beta}_{0})-\frac{1}{\sqrt{2\theta''(z_{0})} (z+z_{0})}\frac{i}{t^{\frac{1}{2}}} ( -\beta_{0}~~\bar{\beta}_{0})\\
   &+\frac{1}{\sqrt{2\theta''(z_{1})}(z-z_{1})} \frac{i}{t^{\frac{1}{2}}} ( \bar{\beta}_{1}~~-\beta_{1})-\frac{1}{\sqrt{2\theta''(z_{1})} (z+z_{1})}\frac{i}{t^{\frac{1}{2}}} ( -\beta_{1}~~\bar{\beta}_{1})+O(t^{-\alpha}),
   \end{split}
\end{align}
for the case $\xi\in(0,2)$, the solution of the RH problem \ref{RH-loc} can be derived as
\begin{align}
\begin{split}
   M^{(loc)}(z)=&(1~~1)+\frac{1}{\sqrt{2\theta''(z_{0})}(z-z_{0})} \frac{i}{t^{\frac{1}{2}}} ( -\beta_{0}~~\bar{\beta}_{0})\\
   &-\frac{1}{\sqrt{2\theta''(z_{0})} (z+z_{0})}\frac{i}{t^{\frac{1}{2}}} ( -\beta_{0}~~\bar{\beta}_{0})+O(t^{-\alpha}).
   \end{split}
\end{align}
In the local domain $\mathcal{U}_{\xi}$, we can obtain the result that
\begin{align}\label{Msp-Est}
|M^{(loc)}-\mathbb{I}|\lesssim O(t^{-\frac{1}{2}}), ~~as~~ t\rightarrow\infty,
\end{align}
which implies that
\begin{align}\label{7-15}
\|M^{(loc)}(z)\|_{\infty}\lesssim 1.
\end{align}

\section{The small-norm RH problem for $E(z)$}

According to the transformation \eqref{decompose-Mrhp}, we have
\begin{align}\label{explict-Ez}
E(z)=\left\{\begin{aligned}
&M^{(2)}_{RHP}(z)M^{R}(z)^{-1}, &&z\in\mathbb{C}\setminus\mathcal{U}_{\xi},\\
&M^{(2)}_{RHP}(z)M^{loc}(z)^{-1}M^{R}(z)^{-1}, &&z\in\mathcal{U}_{\xi},
\end{aligned} \right.
\end{align}
which is analytic in $\mathbb{C}\setminus\Sigma^{(E)}(\xi)$ where $\Sigma^{(E)}(\xi)$ (see Fig. \ref{fig-7}) is defined as
\begin{align*}
\Sigma^{(E)}(\xi)=\left\{\begin{aligned}
&\partial\mathcal{U}_{\xi}\cup(\tilde \Sigma(\xi)\setminus\mathcal{U}_{\xi}),~~\xi\in(-1/4,2),\\
&\tilde \Sigma(\xi), ~~\xi\in(-\infty,-1/4)\cup(2,\infty).
\end{aligned}\right.
\end{align*}

\begin{figure}
	\centering
\subfigure[]{
\begin{tikzpicture}[node distance=3cm]
		\draw[thick](0,0.5)--(3,1)node[above]{$\Sigma_1$};
		\draw[thick](0,0.5)--(-3,1)node[left]{$\Sigma_2$};
		\draw[thick](0,-0.5)--(-3,-1)node[left]{$\Sigma_3$};
		\draw[thick](0,-0.5)--(3,-1)node[right]{$\Sigma_4$};
		\draw[dashed,->](-4,0)--(4,0)node[right]{ Re$z$};
		\draw[dashed,->](0,-2)--(0,2)node[above]{ Im$z$};
		\draw[-latex][thick](0,-0.5)--(-1.5,-0.75);
		\draw[-latex][thick](0,0.5)--(-1.5,0.75);
		\draw[-latex][thick](0,0.5)--(1.5,0.75);
		\draw[-latex][thick](0,-0.5)--(1.5,-0.75);
		\end{tikzpicture}
}
\subfigure[]{
	\begin{tikzpicture}[scale=1.2]
\draw(4.35,0.18)--(5,0.5);
\draw[->](4.35,0.18)--(4.5,0.25);
\draw(3.64,0.15)--(2.5,0.6);
\draw[-<](3.64,-0.15)--(3.25,-0.3);
\draw(4.35,-0.18)--(5,-0.5);
\draw[-<](3.64,0.15)--(3.25,0.3);
\draw(3.64,-0.15)--(2.5,-0.6);
\draw[->](4.35,-0.18)--(4.5,-0.25);
\draw(-4.35,0.18)--(-5,0.5);
\draw[-<](-4.35,0.18)--(-4.5,0.25);
\draw(-3.64,0.15)--(-2.5,0.6);
\draw[->](-3.64,-0.15)--(-3.25,-0.3);
\draw(-4.35,-0.18)--(-5,-0.5);
\draw[->](-3.64,0.15)--(-3.25,0.3);
\draw(-3.64,-0.15)--(-2.5,-0.6);
\draw[-<](-4.35,-0.18)--(-4.5,-0.25);
\draw(-0.64,0.15)--(0,0.5);
\draw[->](-0.64,0.15)--(-0.4,0.28);
\draw(-1.35,0.16)--(-2.5,0.6);
\draw[-<](-1.35,-0.16)--(-1.75,-0.31);
\draw(-0.64,-0.15)--(0,-0.5);
\draw[-<](-1.35,0.16)--(-1.75,0.31);
\draw(-1.35,-0.16)--(-2.5,-0.6);
\draw[->](-0.64,-0.15)--(-0.4,-0.28);
\draw[dashed](-5.5,0)--(5.5,0)node[right]{Re$z$};
\draw [-latex](5.5,0)--(5.6,0);
\draw(0.64,0.15)--(0,0.5);
\draw[-<](0.64,0.15)--(0.4,0.28);
\draw(1.35,0.16)--(2.5,0.6);
\draw[->](1.35,-0.16)--(1.75,-0.31);
\draw(0.64,-0.15)--(0,-0.5);
\draw[->](1.35,0.16)--(1.75,0.31);
\draw(1.35,-0.16)--(2.5,-0.6);
\draw[-<](0.64,-0.15)--(0.4,-0.28);
\draw[->](2.5,0)--(2.5,0.6);
\draw[->](2.5,0)--(2.5,-0.6);
\draw[->](-2.5,0)--(-2.5,0.6);
\draw[->](-2.5,0)--(-2.5,-0.6);
\draw[->](0,0)--(0,0.5);
\draw[->](0,0)--(0,-0.5);
\coordinate (I) at (0,0);
\fill (I) circle (1pt) node[below right] {$0$};
\coordinate (A) at (-4,0);
\fill[blue] (A) circle (1pt) node[below] {$\xi_4$};
\coordinate (b) at (-1,0);
\fill[blue] (b) circle (1pt) node[below] {$\xi_3$};
\coordinate (e) at (4,0);
\fill[blue] (e) circle (1pt) node[below] {$\xi_1$};
\coordinate (f) at (1,0);
\draw[thick,blue](1,0) circle (0.4);
\fill[blue] (f) circle (1pt) node[below] {$\xi_2$};
\draw[thick,blue](4,0) circle (0.4);
\draw[thick,blue](-1,0) circle (0.4);
\draw[thick,blue](-4,0) circle (0.4);
\coordinate (c) at (-2,0);
\fill[red] (c) circle (1pt) node[below] {\scriptsize$-1$};
\coordinate (d) at (2,0);
\fill[red] (d) circle (1pt) node[below] {\scriptsize$1$};
\end{tikzpicture}
}
\subfigure[]{
\begin{tikzpicture}[scale=1.7]
%\draw(4.35,0.18)--(5,0.5);
%\draw[->](4.35,0.18)--(4.5,0.25);
%\draw(3.64,0.15)--(2.5,0.6);
%\draw[-<](3.64,-0.15)--(3.25,-0.3);
%\draw(4.35,-0.18)--(5,-0.5);
%\draw[-<](3.64,0.15)--(3.25,0.3);
%\draw(3.64,-0.15)--(2.5,-0.6);
%\draw[->](4.35,-0.18)--(4.5,-0.25);
%\draw(-4.35,0.18)--(-5,0.5);
%\draw[-<](-4.35,0.18)--(-4.5,0.25);
%\draw(-3.64,0.15)--(-2.5,0.6);
%\draw[->](-3.64,-0.15)--(-3.25,-0.3);
%\draw(-4.35,-0.18)--(-5,-0.5);
%\draw[->](-3.64,0.15)--(-3.25,0.3);
%\draw(-3.64,-0.15)--(-2.5,-0.6);
%\draw[-<](-4.35,-0.18)--(-4.5,-0.25);
\draw(-0.64,0.15)--(0,0.5);
\draw[->](-0.64,0.15)--(-0.4,0.28);
\draw(-1.35,0.16)--(-2.5,0.6);
\draw[-<](-1.35,-0.16)--(-1.75,-0.31);
\draw(-0.64,-0.15)--(0,-0.5);
\draw[-<](-1.35,0.16)--(-1.75,0.31);
\draw(-1.35,-0.16)--(-2.5,-0.6);
\draw[->](-0.64,-0.15)--(-0.4,-0.28);
\draw[dashed](-3.5,0)--(3.5,0)node[right]{Re$z$};
\draw [-latex](3.5,0)--(3.6,0);
\draw(0.64,0.15)--(0,0.5);
\draw[-<](0.64,0.15)--(0.4,0.28);
\draw(1.35,0.16)--(2.5,0.6);
\draw[->](1.35,-0.16)--(1.75,-0.31);
\draw(0.64,-0.15)--(0,-0.5);
\draw[->](1.35,0.16)--(1.75,0.31);
\draw(1.35,-0.16)--(2.5,-0.6);
\draw[-<](0.64,-0.15)--(0.4,-0.28);
%\draw[->](2.5,0)--(2.5,0.6);
%\draw[->](2.5,0)--(2.5,-0.6);
%\draw[->](-2.5,0)--(-2.5,0.6);
%\draw[->](-2.5,0)--(-2.5,-0.6);
\draw[->](0,0)--(0,0.5);
\draw[->](0,0)--(0,-0.5);
\coordinate (I) at (0,0);
\fill (I) circle (1pt) node[below right] {$0$};
%\coordinate (A) at (-4,0);
%\fill[blue] (A) circle (1pt) node[below] {$\xi_4$};
\coordinate (b) at (-1,0);
\fill[blue] (b) circle (1pt) node[below] {$\xi_3$};
%\coordinate (e) at (4,0);
%\fill[blue] (e) circle (1pt) node[below] {$\xi_1$};
\coordinate (f) at (1,0);
\draw[thick,blue](1,0) circle (0.4);
\fill[blue] (f) circle (1pt) node[below] {$\xi_2$};
%\draw[thick,red](4,0) circle (0.4);
\draw[thick,blue](-1,0) circle (0.4);
%\draw[thick,red](-4,0) circle (0.4);
\coordinate (c) at (-2,0);
\fill[red] (c) circle (1pt) node[below] {\scriptsize$-1$};
\coordinate (d) at (2,0);
\fill[red] (d) circle (1pt) node[below] {\scriptsize$1$};
\end{tikzpicture}
}
\caption{ ($a$) and ($b$) denote the jump contour $\Sigma^{(E)}$ as $\xi\in(-\infty,-1/4)\cup(2,\infty)$, $\xi\in(-1/4,0)$ and $\xi\in(0,2)$, respectively. The blue circles represent the jump contour $U(\xi)$ at each phase points $\xi_j$.}\label{fig-7}
\end{figure}
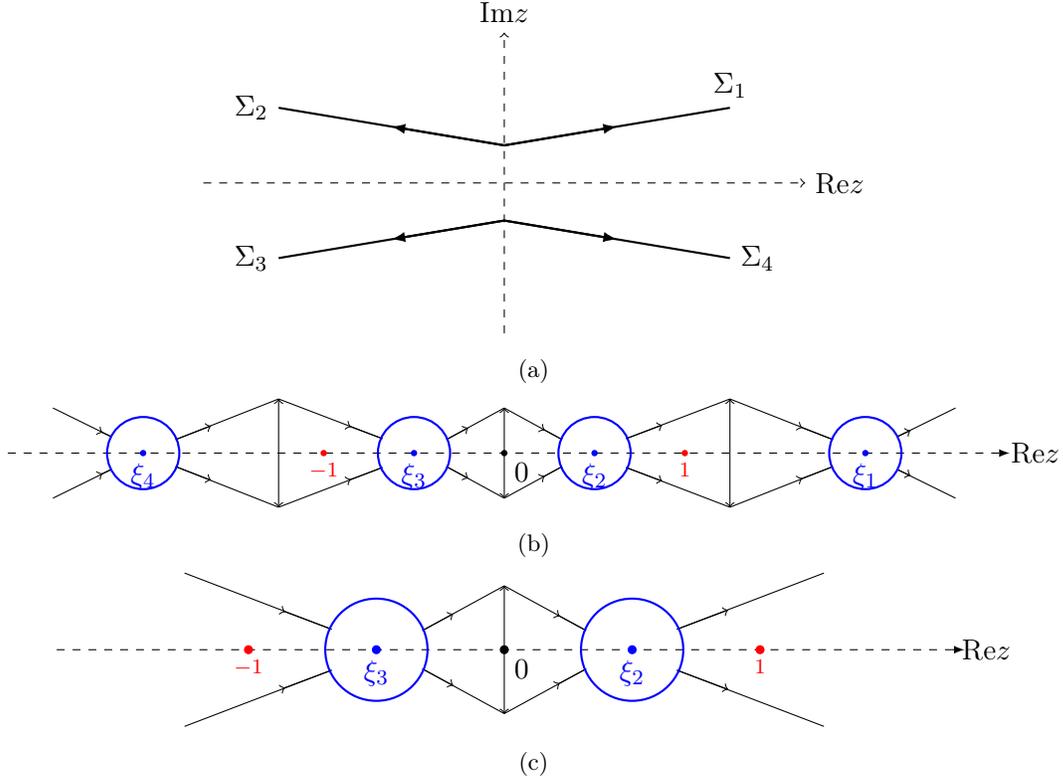

Then, we obtain a RH problem for  $E(z)$.
\begin{RHP}\label{RH-9}
Find a matrix-valued function $E(y,t,z)$ such that
\begin{itemize}
\item $E(y,t,z)$ is analytic in $\mathbb{C}\setminus\Sigma^{(E)}(\xi)$;
\item $E(y,t;z)\to(1~~1)$ as $z\rightarrow \infty$;
\item $E_+(y,t,z)=E_-(y,t,z)V^{(E)}(z)$, \quad $z\in\Sigma^{(E)}(\xi)$, where
\end{itemize}
\begin{align}\label{7-17}
V^{(E)}(z)=\left\{\begin{aligned}
&M^{R}(z)V^{(2)}(z)M^{R}(z)^{-1}, &&z\in\tilde\Sigma(\xi) \setminus \mathcal{U}_{\xi},~~or~~z\in\tilde\Sigma(\xi)\\
&M^{R}(z)M^{(loc)}(z)M^{R}(z)^{-1}, &&z\in\partial\mathcal{U}_{\xi}.
\end{aligned}\right.
\end{align}
\end{RHP}

Next, we evaluate the estimate of the jump matrix $V^{(E)}(z)$.

Based on Propositions \ref{v2-estimate-xi-1},\ref{v2-estimate-xi-2}  and the boundedness of $M^{R}(z)$, as $t\to \infty$, we have
\begin{align}\label{VE-I-1}
|V^{(E)}(z)-\mathbb{I}|=\left\{\begin{aligned}
&O\left(e^{-2\varepsilon_0c_3t}\right) &&z\in\tilde\Sigma(\xi\in(-\infty,-1/4)\cup(2,\infty)),\\
&O\left(e^{-c_4t}\right) &&z\in\tilde\Sigma(\xi\in(-1/4,2)\setminus\mathcal{U}_{\xi},\\
&O\left(e^{-c'_4t}\right) &&z\in\Sigma'_{j\pm}.
\end{aligned}\right.
\end{align}
For $z\in\partial\mathcal{U}_{\xi}$, based on the boundedness of $M^{R}(z)$ and the results of $M^{(loc)}(z)$, we obtain
\begin{align}\label{VE-I-2}
  |V^{(E)}(z)-\mathbb{I}|=|M^{R}(z)(M^{(loc)}(z)-\mathbb I)M^{R}(z)^{-1}|=O(|t|^{-\frac{1}{2}}).
\end{align}

Then, by using a small-norm RH problem, the existence and uniqueness of RH problem \ref{RH-9} can be guaranteed.  Furthermore,  on the basis of Beals-Coifman theory, we obtain that
\begin{align}\label{Ez-solution}
E(z)=\mathbb{I}+\frac{1}{2\pi i}\int_{\Sigma^{(E)}(\xi)}\frac{(\mathbb{I}+\mu_E(s))(V^{(E)}(s)-\mathbb{I})}{s-z}\,ds,
\end{align}
where $\mu_E\in L^2 (\Sigma^{(E)}) $ and satisfies
\begin{align}\label{7-18}
(1-C_{\omega_E})\mu_E=\mathbb{I},
\end{align}
where $C_{\omega_E}$ is an integral operator and  is defined as
\begin{align*}
C_{\omega_E}f=C_{-}\left(f(V^{(E)}-\mathbb{I})\right),\\
C_{-}f(z)=\lim_{z\rightarrow\Sigma_{-}^{(E)}(\xi)}\frac{1}{2\pi i}\int_{\Sigma^{(E)}(\xi)}\frac{f(s)}{s-z}\,ds,
\end{align*}
where $C_{-}$ is the Cauchy projection operator.
Next, according to the properties of the Cauchy projection operator $C_{-}$, and the estimate \eqref{VE-I-1} and \eqref{VE-I-2}, we have
\begin{align}%\label{}
\|C_{\omega_E}f\|_{L^2(\Sigma^{(E)}(\xi))}\lesssim\|C_-\|_{L^2(\Sigma^{(E)}(\xi))} \|f\|_{L^2(\Sigma^{(E)}(\xi))}
\|V^{(E)}-\mathbb{I}\|_{L^{\infty}(\Sigma^{(E)}(\xi))}\lesssim O(t^{-1/2}),
\end{align}
which implies $1-C_{\omega_E}$ is invertible. Then, the existence and uniqueness of $\mu_E$ is established. As a result,  the existence and uniqueness of $E(z)$ are guaranteed. These facts show  that the definition of $M^{(2)}_{RHP}$ is reasonable.

Furthermore, to reconstruct the solutions of $q(x,t)$, we need to study the asymptotic behavior of $E(z)$ as $z\rightarrow\frac{i}{2}$ and large time asymptotic behavior of $E(z)$. On the basis of \eqref{VE-I-1} and \eqref{VE-I-2}, as $t\rightarrow\infty$, we only need to consider the calculation on $\partial\mathcal{U}_{\xi}$ because it approaches to zero exponentially on other boundaries. Then, as $z\rightarrow\frac{i}{2}$, we can obtain that
\begin{align}\label{7-19}
E(z)=E(\frac{i}{2})+E_{1}(z-\frac{i}{2})+O((z-\frac{i}{2})^{-2}),
\end{align}
where
\begin{gather}
E(\frac{i}{2})=\mathbb{I}+\frac{1}{2\pi i}\int_{\Sigma^{(E)(\xi)}(\xi)}\frac{(\mathbb{I}+\mu_E(s))(V^{(E)}(s)-\mathbb{I})} {s-\frac{i}{2}}\,ds,\\
E_{1}=-\frac{1}{2\pi i}\int_{\Sigma^{(E)}(\xi)}\frac{(\mathbb{I}+\mu_E(s))(V^{(E)}(s)-I)} {(s-\frac{i}{2})^2}\,ds.\label{7-22}
\end{gather}
Then, the long time, i.e., $t\rightarrow\infty$, asymptotic behavior  of $E(\frac{i}{2})$  and $E_{1}$ can be derived.\\
$\bullet$ For $\xi\in(-1/4,0)$, $E(\frac{i}{2})$ can be expressed as
\begin{align}
  E(\frac{i}{2})=(1~~1)+\frac{1}{t^{\frac{1}{2}}}B^{(0)}(\xi)+O(t^{-1}),
\end{align}
where
\begin{align*}
  B^{(0)}(\xi)=&\frac{1}{(z_{0}-\frac{i}{2})} \frac{i}{t^{\frac{1}{2}}} M^{R}(z_0)^{-1}( -\beta_{0}~~\bar{\beta}_{0})M^{R}(z_0)\\
  &-\frac{1}{(-z_{0}-\frac{i}{2})}\frac{i}{t^{\frac{1}{2}}} M^{R}(-z_0)^{-1}( -\beta_{0}~~\bar{\beta}_{0})M^{R}(-z_0)\\
   &+\frac{1}{(z_{1}-\frac{i}{2})} \frac{i}{t^{\frac{1}{2}}} M^{R}(z_1)^{-1}( \bar{\beta}_{1}~~-\beta_{1})M^{R}(z_1)\\
  & -\frac{1}{(-z_{1}-\frac{i}{2})}\frac{i}{t^{\frac{1}{2}}} M^{R}(-z_1)^{-1}( -\beta_{1}~~\bar{\beta}_{1})M^{R}(-z_1)+O(t^{-1}),
\end{align*}
and $E_{1}$ can be expressed as
\begin{align}\label{B-1}
  E_{1}=\frac{1}{t^{\frac{1}{2}}}B^{(1)}(\xi)+O(t^{-1}),
\end{align}
where
\begin{align*}
  B^{(1)}(\xi)=&-\frac{1}{(z_{0}-\frac{i}{2})^{2}} \frac{i}{t^{\frac{1}{2}}} M^{R}(z_0)^{-1}( -\beta_{0}~~\bar{\beta}_{0})M^{R}(z_0)\\
  &+\frac{1}{(-z_{0}-\frac{i}{2})^{2}}\frac{i}{t^{\frac{1}{2}}} M^{R}(-z_0)^{-1}( -\beta_{0}~~\bar{\beta}_{0})M^{R}(-z_0)\\
   &-\frac{1}{(z_{1}-\frac{i}{2})^{2}} \frac{i}{t^{\frac{1}{2}}} M^{R}(z_1)^{-1}( \bar{\beta}_{1}~~-\beta_{1})M^{R}(z_1)\\
   &+\frac{1}{(-z_{1}-\frac{i}{2})^{2}}\frac{i}{t^{\frac{1}{2}}} M^{R}(-z_1)^{-1}( -\beta_{1}~~\bar{\beta}_{1})M^{R}(-z_1)+O(t^{-1}),
\end{align*}
$\bullet$ For $\xi\in(0,2)$, can be expressed as
\begin{align}
  E(\frac{i}{2})=(1~~1)+\frac{1}{t^{\frac{1}{2}}}B^{(0)}(\xi)+O(t^{-1}),
\end{align}
where
\begin{align*}
  B^{(0)}(\xi)=&\frac{1}{(z_{0}-\frac{i}{2})} \frac{i}{t^{\frac{1}{2}}} M^{R}(z_0)^{-1}( -\beta_{0}~~\bar{\beta}_{0})M^{R}(z_0)\\
  &-\frac{1}{(-z_{0}-\frac{i}{2})}\frac{i}{t^{\frac{1}{2}}} M^{R}(-z_0)^{-1}( -\beta_{0}~~\bar{\beta}_{0})M^{R}(-z_0)+O(t^{-1}),
\end{align*}
and $E_{1}$ can be expressed as
\begin{align}
  E_{1}=\frac{1}{t^{\frac{1}{2}}}B^{(1)}(\xi)+O(t^{-1}),
\end{align}
where
\begin{align*}
  B^{(1)}(\xi)=&-\frac{1}{(z_{0}-\frac{i}{2})^2} \frac{i}{t^{\frac{1}{2}}} M^{R}(z_0)^{-1}( -\beta_{0}~~\bar{\beta}_{0})M^{R}(z_0)\\
  &+\frac{1}{(-z_{0}-\frac{i}{2})^2}\frac{i}{t^{\frac{1}{2}}} M^{R}(-z_0)^{-1}( -\beta_{0}~~\bar{\beta}_{0})M^{R}(-z_0)+O(t^{-1}).
\end{align*}

\section{Pure $\bar{\partial}$-RH problem}\label{section-Pure-dbar-RH}

In this section, we will pay attention to the remaining $\bar{\partial}$-RH problem to show that it has a solution and bound its size.

The $\bar{\partial}$-RH problem \ref{RH-4} for $M^{(3)}(z)$ is equivalent to the following integral equation
\begin{align}\label{8.1}
M^{(3)}(z)=\mathbb{I}-\frac{1}{\pi}\int_{\mathbb{C}}\frac{M^{(3)}W^{(3)}}{s-z}\mathrm{d}A(s),
\end{align}
where $\mathrm{d}A(s)$ is Lebesgue measure. Furthermore, the equation \eqref{8.1} can be written in an operator form
\begin{align}\label{8.2}
(\mathbb{I}-\mathrm{S})M^{(3)}(z)=\mathbb{I},
\end{align}
where $\mathrm{S}$ is Cauchy operator
\begin{align}\label{8.3}
\mathrm{S}[f](z)=-\frac{1}{\pi}\iint_{\mathbb{C}}\frac{f(s)W^{(3)}(s)}{s-z}\mathrm{d}A(s).
\end{align}
We need to prove that the  operator $(\mathbb{I}-\mathrm{S})$ is invertible so that the solution $M^{(3)}(z)$ exists.
While according to the previous analysis, we know that  $W^{(3)} (s)$ possesses different properties and structures for the case $\xi\in(-\infty,-\frac{1}{4})\cup(2,\infty)$ and $\xi\in(-\frac{1}{4},2)$. Therefore, we   need to consider it respectively.

\subsection{Region $\xi\in(-\infty,-\frac{1}{4})\cup(2,\infty)$}

Firstly, we give the proof of the existence of operator $(\mathbb{I}-\mathrm{S}^{-1}$.
\begin{lem}\label{lem-1-1}
For $t\rightarrow\infty$, the operator \eqref{8.3} admits that
\begin{align}\label{8.4}
||\mathrm{S}||_{L^{\infty}\rightarrow L^{\infty}}\leq ct^{-1/2},
\end{align}
where $c$ is a constant.
\end{lem}
\begin{proof}
We mainly prove the case that the matrix function supported in the region $\Omega_1$, the others can be proved similarly. Denote $f\in L^{\infty}(\Omega_1)$, $s=u+iv$ and $z=x+iy$. Then based on \eqref{dbar-R2-1} and \eqref{RH-4-dbar}, we can derive that
\begin{align}\label{8.5}
|S[f](z)|&\leq\frac{1}{\pi}\big|f\ \big|_{L^{\infty}(\Omega_{1})}
\iint_{\Omega_{1}}\frac{|M^{(2)}_{RHP}(s)\bar{\partial}R_{1}(s)M^{(2)}_{RHP}(s)^{-1}|}
{|s-z|}df(s)\notag\\
&\leq c\iint_{\Omega_{1}}
\frac{|\bar{\partial}R_{1}(s)||e^{-2it\theta(z)}}{|s-z|}dudv,
\end{align}
where $c$ is a constant.

Based on \eqref{R-estimate}, from \eqref{8.5}, we obtain that
\begin{align}\label{8.6}
||\mathrm{S}||_{L^{\infty}\rightarrow L^{\infty}}\leq c(I_{1}+I_{2})\leq ct^{-1/2},
\end{align}
where
\begin{align}\label{8.8}
I_{1}=\iint_{\Omega_{1}}
\frac{|z-i\varepsilon_0|^{-1/2}
e^{2tIm\theta(z)}}{|s-z|}df(s), ~~
I_{2}=\iint_{\Omega_{1}}
\frac{|p'(u)|e^{2tIm\theta(z)}}{|s-z|}df(s).
\end{align}
By using a fact that
\begin{align}\label{s-z-0}
\Big|\Big|\frac{1}{s-z}\Big|\Big|_{L^{2}(v+z_{0},\infty)}=\left(\int_{v+z_{0}}^{\infty}
\frac{1}{|s-z|^{2}}du\right)^{\frac{1}{2}}
\leq\frac{\pi}{v-y},
\end{align}
we can prove that $I_{1},I_{2}\lesssim t^{-1/2}$.
\end{proof}

Next, our purpose is to reconstruct the large time asymptotic behaviors of $q(x,t)$. According to \eqref{solution-2}, we need the large time asymptotic behaviors of  $M^{(3)}(\frac{i}{2})$  and $M_{1}^{(3)}(y,t)$ which are defined as
\begin{align*}
M^{(3)}(z)=M^{(3)}(\frac{i}{2})+M_{1}^{(3)}(y,t)(z-\frac{i}{2}) +O((z-\frac{i}{2})^{2}),~~z\rightarrow\frac{i}{2} ,
\end{align*}
where
\begin{align*}
M^{(3)}(\frac{i}{2})=\mathbb{I}-\frac{1}{\pi}\iint_{\mathbb{C}} \frac{M^{(3)}(s)W^{(3)}(s)}{s-\frac{i}{2}}
\mathrm{d}A(s),\\
M^{(3)}_{1}(y,t)=\frac{1}{\pi}\int_{\mathbb{C}} \frac{M^{(3)}(s)W^{(3)}(s)}{(s-\frac{i}{2})^{2}}
\mathrm{d}A(s).
\end{align*}
The $M^{(3)}(\frac{i}{2})$ and $M^{(3)}_{1}(y,t)$ satisfy the following lemma.
\begin{lem}\label{prop-M3-Est}
For $t\rightarrow+\infty$, there exists a positive constant $\tau<\frac{1}{4}$ such that $M^{(3)}(\frac{i}{2})$ and $M^{(3)}_{1}(y,t)$ admit the following inequality
\begin{align}
\|M^{(3)}(\frac{i}{2})-\mathbb{I}\|_{L^{\infty}}\lesssim t^{-1+2\tau},\label{8.10}\\
M^{(3)}_{1}(y,t)\lesssim t^{-1+2\tau}\label{8.11}.
\end{align}
\end{lem}
The proof of this Lemma is similar to Lemma \ref{lem-1-1}.

\subsection{Region $\xi\in(-\frac{1}{4},2)$}

The first step is still to prove the existence of operator $(\mathbb{I}-\mathrm{S}^{-1}$.
\begin{lem}\label{lem-1-2}
For $t\rightarrow\infty$, the operator \eqref{8.3} admits that
\begin{align}\label{8.4+1}
||\mathrm{S}||_{L^{\infty}\rightarrow L^{\infty}}\leq ct^{-1/2},
\end{align}
where $c$ is a constant.
\end{lem}
\begin{proof}
We mainly prove the case that the matrix function supported in the region $\Omega_{11}$, the others can be proved similarly. Denote $f\in L^{\infty}(\Omega_{11})$, $s=u+iv$ and $z=x+iy$. Then based on \eqref{dbar-R2-3} and \eqref{RH-4-dbar}, we can derive that
\begin{align}\label{8.5+1}
|S[f](z)|&\leq\frac{1}{\pi}\big|f\ \big|_{L^{\infty}(\Omega_{11})}
\iint_{\Omega_{11}}\frac{|M^{(2)}_{RHP}(s)\bar{\partial}R_{1}(s)M^{(2)}_{RHP}(s)^{-1}|}
{|s-z|}df(s)\notag\\
&\leq c\iint_{\Omega_{11}}
\frac{|\bar{\partial}R_{11}(s)||e^{-2it\theta(z)}}{|s-z|}dudv,
\end{align}
where $c$ is a constant.

Based on \eqref{R-estimate-2}, from \eqref{8.5+1}, we obtain that
\begin{align}\label{8.6+1}
||\mathrm{S}||_{L^{\infty}\rightarrow L^{\infty}}\leq c(\tilde{I}_{1}+\tilde{I}_{2})\leq ct^{-1/2},
\end{align}
where
\begin{align}\label{8.8+1}
\tilde{I}_{1}=\iint_{\Omega_{11}}
\frac{|z-\xi_1|^{-1/2}
e^{2tIm\theta(z)}}{|s-z|}df(s), ~~
\tilde{I}_{2}=\iint_{\Omega_{11}}
\frac{|p'(u)|e^{2tIm\theta(z)}}{|s-z|}df(s).
\end{align}
By using \eqref{s-z-0},
we can prove that $\tilde{I}_{1},\tilde{I}_{2}\lesssim t^{-1/2}$.
\end{proof}

Next, our purpose is to reconstruct the large time asymptotic behaviors of $q(x,t)$. According to \eqref{solution-2}, we need the large time asymptotic behaviors of  $M^{(3)}(\frac{i}{2})$  and $M_{1}^{(3)}(y,t)$ which are defined as
\begin{align*}
M^{(3)}(z)=M^{(3)}(\frac{i}{2})+M_{1}^{(3)}(y,t)(z-\frac{i}{2}) +O((z-\frac{i}{2})^{2}),~~z\rightarrow\frac{i}{2} ,
\end{align*}
where
\begin{align*}
M^{(3)}(\frac{i}{2})=\mathbb{I}-\frac{1}{\pi}\iint_{\mathbb{C}} \frac{M^{(3)}(s)W^{(3)}(s)}{s-\frac{i}{2}}
\mathrm{d}A(s),\\
M^{(3)}_{1}(y,t)=\frac{1}{\pi}\int_{\mathbb{C}} \frac{M^{(3)}(s)W^{(3)}(s)}{(s-\frac{i}{2})^{2}}
\mathrm{d}A(s).
\end{align*}
The $M^{(3)}(\frac{i}{2})$ and $M^{(3)}_{1}(y,t)$ satisfy the following lemma.
\begin{lem}\label{prop-M3-Est-2}
For $t\rightarrow+\infty$,  $M^{(3)}(\frac{i}{2})$ and $M^{(3)}_{1}(y,t)$ admit the following inequality
\begin{align}
\|M^{(3)}(\frac{i}{2})-\mathbb{I}\|_{L^{\infty}}\lesssim t^{-1},\label{8-10+1}\\
M^{(3)}_{1}(y,t)\lesssim t^{-1}\label{8-11+1}.
\end{align}
\end{lem}
The proof of this Lemma is similar to Lemma \ref{lem-1-2}.

\section{Asymptotic approximation  for the CH equation}\label{Asy-appro-CH}

Now, we are going to construct the long-time asymptotic of the CH equation \eqref{CH-equation}.
Recall a series of transformations including \eqref{Trans-1},\eqref{Trans-1+1},\eqref{Trans-2},\eqref{delate-pure-RHP} and \eqref{decompose-Mrhp}, i.e.,
\begin{align*}
M(z)\leftrightarrows M^{(1)}(z)\leftrightarrows M^{(2)}(z)\leftrightarrows M^{(3)}(z) \leftrightarrows E(z),
\end{align*}
we then obtain
\begin{align*}
M(z)=M^{(3)}(z)E(z)M^{R}(z)(R^{(2)}(z))^{-1}T^{\sigma_{3}}(z),~~ z\in\mathbb{C}\setminus\mathcal{U}_{\xi}.
\end{align*}
In order to recover the solution $q(x,t)$ , we take $z\rightarrow\frac{i}{2}$ along the imaginary axis, which implies $z\in\mathbb C\setminus\bar{\Omega}$ thus $R^{(2)}(z)=\mathbb I$. Then, as $z\to\frac{i}{2}$,  we obtain
\begin{align*}
M(z)=&\left(M^{(3)}(i/2)+M^{(3)}_1(z-i/2)\right)E(z,\xi)M^{R}_{j_0}(z)\\
&\times T^{\sigma_{3}}(\frac{i}{2})(1+(z-\frac{i}{2})T^{\sigma_{3}}_{1})+O((z-i/2)^2),
\end{align*}
where $T_{1}$ is defined in \eqref{4.10}.
As $t\to\infty$,  the long time asymptotic behavior of $M(z)$ can be derived.\\
$\bullet$ For $\xi\in\left(-\infty,-\frac{1}{4}\right)\cup(2,\infty)$, we have
\begin{align*}
M(z)=M^{R}_{j_0}(z)T(\frac{i}{2})(1+(z-\frac{i}{2})T_{1})+O(t^{-1+2\tau}),
\end{align*}
and
\begin{align*}
M(i/2)=M^{R}_{j_0}(i/2)T^{\sigma_{3}}(\frac{i}{2})+O(t^{-1+2\tau}).
\end{align*}

Furthermore, based on  \eqref{solution-2}, we can derive that
\begin{align*}
q(x,t)=&q(y(x,t),t)=\frac{1}{2i}\lim_{z\to\frac{i}{2}}\left(\frac{M_1(y,t;z)M_2(y,t;z)} {M_1(y,t;\frac{i}{2})M_2(y,t;\frac{i}{2})}-1\right)\frac{1}{z-\frac{i}{2}}\\
=&\frac{1}{2i}\lim_{z\to\frac{i}{2}}\left(\frac{M^{R}_{j_0,1}(y,t;z)M^{R}_{j_0,2}(y,t;z)} {M^{R}_{j_0,1}(y,t;\frac{i}{2})M^{R}_{j_0,2}(y,t;\frac{i}{2})}-1\right) \frac{1}{z-\frac{i}{2}}+O(t^{-1+2\tau}),\\
=&q(x,t|D_{j_0})+O(t^{-1+2\tau}),\\
x(y,t)&=y+\ln\frac{M_1(y,t;\frac{i}{2})}{M_2(y,t;\frac{i}{2})}\\
&=y+2\ln T(i/2)+c(x,t|D_{j_0})+O(t^{-1+2\tau}),
\end{align*}
where $q(x,t|D_{j_0})$ and $c(x,t|D_{j_0})$ are defined in \eqref{3-solution-1}.\\
$\bullet$ For $\xi\in\left(-\frac{1}{4},2\right)$, we have
\begin{align*}
M(z)=(E(i/2)+E_1(z-i/2))M^{R}_{j_0}(z)T^{\sigma_{3}}(\frac{i}{2}) (1+(z-\frac{i}{2})T_{1})^{\sigma_{3}}+O(t^{-1}),
\end{align*}
Then, based on  \eqref{solution-2}, we can derive that
\begin{align*}
q(x,t)=&q(y(x,t),t)=\frac{1}{2i}\lim_{z\to\frac{i}{2}}\left(\frac{M_1(y,t;z)M_2(y,t;z)} {M_1(y,t;\frac{i}{2})M_2(y,t;\frac{i}{2})}-1\right)\frac{1}{z-\frac{i}{2}}\\
=&q(x,t|D_{j_0})+f_1t^{-1/2}+O(t^{-1}),\\
x(y,t)&=y+2\ln\frac{M_1(y,t;\frac{i}{2})}{M_2(y,t;\frac{i}{2})}\\
&=y+2\ln T(i/2)+c(x,t|D_{j_0})+f_2t^{-1/2}+O(t^{-1}),
\end{align*}
where $q(x,t|D_{j_0})$ and $c(x,t|D_{j_0})$ are defined in \eqref{3-solution-1}, and $f_1$ and $f_2$ are respectively defined as
\begin{align}\label{define-f-1-2}
 \begin{split}
   f_1=&B^{(1)}_1+B^{(1)}_2,\\
   f_2=&\frac{M^{R}_{j_0,1}B^{(1)}_1}{M^{R}_{j_0,2}B^{(1)}_2}.
 \end{split}
\end{align}
In \eqref{define-f-1-2}, $B^{(1)}_j$ means the $j$-th row of  $B^{(1)}$ and $B^{(1)}$ is defined in \eqref{B-1}.

Finally, we give the main results of this work.

\begin{thm}\label{Thm-1}
Suppose that the initial value $q_{0}(x)\in H^{4,2}(\mathbb R)$. Let $q(x,t)$ be the solution of CH equation \eqref{CH-equation}. The scattering data is denoted as $\mathcal{Z}=\{z_j, \gamma_j\}_{j=1}^{N}(Re z_j=0,~0<Im z_j<\frac{1}{2},~\gamma_j>0)$  generated from the initial values $q_{0}(x)$. Let $\xi=\frac{y}{t}$ and $q(x,t|D_{j_0})$ be the $N(j_0)$-soliton solution corresponding to scattering data $D_{j_0}=\{r(z)\equiv0, \{z_{j_0}, \gamma_{j_0}T^{-2}(z_{j_0})\}$. $j_0$ is defined in \eqref{j-0}. Then, as $t\to\infty$, we have the following results:
\begin{itemize}
  \item For $\xi\in\left(-\infty,-\frac{1}{4}\right)\cup(2,\infty)$,
  \begin{align}\label{result-1}
q(x,t)=&q(y(x,t),t)=q(x,t|D_{j_0})+O(t^{-1+2\tau}),\\
x(y,t)&=y+2\ln T(i/2)+c(x,t|D_{j_0})+O(t^{-1+2\tau}),
\end{align}
where $q(x,t|D_{j_0})$ and $c(x,t|D_{j_0})$ are defined in \eqref{3-solution-1}, $T(z)$ is shown in Proposition \ref{T-property}.
  \item For $\xi\in\left(-\frac{1}{4},2\right)$,
  \begin{align}\label{result-2}
q(x,t)=&q(y(x,t),t)=q(x,t|D_{j_0})+f_1t^{-1/2}+O(t^{-1}),\\
x(y,t)&=y+2\ln T(i/2)+c(x,t|D_{j_0})+f_2t^{-1/2}+O(t^{-1}),
\end{align}
where $q(x,t|D_{j_0})$ and $c(x,t|D_{j_0})$ are defined in \eqref{3-solution-1}, and $f_1$ and $f_2$ are defined in \eqref{define-f-1-2}.
\end{itemize}
\end{thm}

\begin{rem}
Theorem \ref{Thm-1}  need the condition $q_{0}(x)\in H^{4,2}(\mathbb{R})$ so that the inverse scattering transform possesses  well mapping properties \cite{r-bijectivity}. Moreover, %it is noted that
the asymptotic results only depend on the $H^{4,2}(\mathbb{R})$ norm of $r$ in this work. So we restrict the initial potential $q_{0}(x)\in H^{4,2}(\mathbb{R})$.
\end{rem}

%\begin{rem}
%The steps in the steepest descent analysis of RHP \ref{RH-1} for $t\rightarrow-\infty$ is similar to the case  $t\rightarrow+\infty$. When we consider $t\rightarrow-\infty$, the main difference can be traced back to the fact that the regions of growth and decay of the exponential factors $e^{2it\theta}$ are reversed, see Fig. \ref{fig-1}. Here, we leave the detailed calculations to the interested readers.
%\end{rem}

The above results show that soliton solutions associated with the poles near critical
lines have  significant contribution to the solution of CH equation as $t\to\infty$. The long-time asymptotic behavior \eqref{result-1} and \eqref{result-2} give the soliton resolution conjecture for the initial value problem of the CH equation.\\

\noindent{\bf Acknowledgments}
%The authors would like to thank the editor and refree for their valuable comments and suggestions.
The authors  would like to express their gratitude to Professor  Engui Fan  for his  helpful discussions and useful comments on this work.
This work was supported by the National Natural Science Foundation of China under Grant No. 11975306, the Natural Science Foundation of Jiangsu Province under Grant No. BK20181351, the Six Talent Peaks Project in Jiangsu Province under Grant No. JY-059,
the 333 Project in Jiangsu Province, and the Fundamental Research Fund for the Central Universities under the Grant No. 2019ZDPY07.
\\
%\noindent\textbf{Compliance with ethical standards}\\

%\noindent\textbf{Conflict of interest} The authors declare that they have no conflict of interest.

%\section*{Appendix A: Detailed calculations for the pure $\bar{\partial}$-Problem  }

\bibliographystyle{plain}

\end{document}